\documentclass[preprint,12pt,times]{article}
\usepackage{graphicx}
\usepackage{amsthm}
\usepackage{amssymb}
\usepackage{amsmath}
\usepackage{amscd}
\usepackage{bbm}
\usepackage{float}
\usepackage{booktabs}
\usepackage{relsize}
\usepackage{mathrsfs}
\numberwithin{equation}{section}

\newtheorem{theorem}{Theorem}[section]
\newtheorem{lemma}{Lemma}[section]
\newtheorem{remark}{Remark}[section]

\newtheorem{proposition}{Proposition}[section]

\usepackage[numbers,sort&compress]{natbib}

\numberwithin{figure}{section}
\numberwithin{table}{section}
\newtheorem{example}{Example}

\allowdisplaybreaks[2]

\newcommand\btd{\raise 2pt \hbox{$\hat\bigtriangledown$}\hskip 1.5pt}
\newcommand\bt{\raise 2pt \hbox{$\bigtriangledown$}\hskip 1.5pt}

\newcommand{\ud}{\mathrm{d}}

\begin{document}
	\date{}
	\title{ Harmonic  maps  and 2D Boussinesq equations }

	\author{Jian Li $^{a,c}$, ~Shaojie Yang $^{a,b}$\thanks{Corresponding author: shaojieyang@kust.edu.cn (Shaojie Yang)} \\~\\
		\small$^a$ Department of  Mathematics,~~Kunming University of Science and Technology,  \\
		\small Kunming, Yunnan 650500, China\\~\\
		\small$^b$ Research Center for Mathematics and Interdisciplinary Sciences,\\
		\small Kunming University of Science and Technology,\\
		\small Kunming, Yunnan 650500, China\\~\\
	\small$^c$ Institute for Advanced Study, Shenzhen University, \\
	\small Shenzhen 518060, China}

	\date{}
	\maketitle
	\begin{abstract}
	 Within the framework of Lagrangian variables, we develop a method for deriving explicit solutions to the 2D Boussinesq equations using harmonic mapping theory. By reformulating the characterization of flow solutions described by harmonic functions, we reduce the problem to solving a particular nonlinear differential system in complex space  \(\mathbb{C}^ 4 \).  To solve  this nonlinear differential system, we introduce the Schwarzian and pre-Schwarzian derivatives, and derive the properties of the sense-preserving harmonic mappings with equal Schwarzian and pre-Schwarzian derivatives.
Our method yields explicit solutions in Lagrangian coordinates that contain two fundamental classes of classical solutions.: Kirchhoff's elliptical vortex (1876) and Gerstner's gravity wave (1809, rediscovered by Rankine in 1863).	\\
		
		\noindent\emph{Keywords}: Boussinesq equations; Harmonic maps; Lagrangian variables\\
		
		\noindent\emph{Mathematics Subject Classification}:  76B03 · 35Q31 · 76M40
	\end{abstract}
	\noindent\rule{15.5cm}{0.5pt}

\newpage	
\tableofcontents

	\section{Introduction}
The prevalence of turbulence throughout the universe is evidenced by the multi-scale dynamics observed in nearly all astrophysical plasma flows, which exhibit diverse spatial and temporal characteristics. This turbulent behavior manifests similarly in Earth's atmospheric and oceanic systems: atmospheric turbulence arises from small-scale chaotic air movements driven by wind patterns, while ocean circulation displays turbulent features across multiple scales \cite{r17}. Within specific scale ranges of both atmospheric and oceanic systems, fluid dynamics becomes governed by the interaction between gravitational forces and planetary rotation, coupled with density fluctuations relative to a reference state \cite{R2,r16}. The Boussinesq equations, recognized as a fundamental geophysical model, effectively capture these convective processes occurring in oceanic and atmospheric systems at these characteristic scales \cite{R2,R3}.	
 The  2D  Boussinesq equations can be written as 
	\begin{align}\label{A1}
	 	\begin{cases}
	 		u_t+uu_x+vu_y+P_x=\mu\Delta u,\\
	 		v_t+uv_x+vv_y+P_y=\mu\Delta v+\theta,\\
	 		\theta_t+u \theta_x+v\theta_y=\kappa\Delta\theta,\\
			u_x+v_y=0,
	 	\end{cases}
	 \end{align}
where $\left(u(t,x,y),v(t,x,y)\right)$, $P=P(t,x,y)$, $\theta=\theta(t,x,y)$, \(\mu\) and \(\kappa,\)    respectively, represent the 2D fluid velocity, the  pressure, the temperature in the content of thermal convection and the density in the modeling of geophysical fluids, the viscosity, the thermal diffusivity.	
	
	From the viewpoint of mathematics, the   Boussinesq equations   has been attracted considerable attention in the past years since it is closely related to the incompressible Euler equations. When $\mu=0$ and $\theta=0$, the  2D  Boussinesq equations \eqref{A1} becomes  2D incompressible   Euler equations	
	\begin{align}\label{A2}
	 	\begin{cases}
	 		u_t+uu_x+vu_y+P_x=0,\\
	 		v_t+uv_x+vv_y+P_y=0,\\
			u_x+v_y=0.
	 	\end{cases}
	 \end{align}
	
	In fluid dynamics, the Lagrangian framework provides the most comprehensive description of flow behavior by tracking individual fluid particles along their trajectories. To analyze fluid systems qualitatively, researchers often study perturbations of known exact solutions—making non-trivial, explicit solutions that accurately reflect key physical phenomena critically valuable. Such solutions serve as foundational benchmarks for understanding complex flow dynamics, validating numerical models, and revealing underlying mechanistic principles \cite{R4}. 	The study of localized vorticity solutions to the two-dimensional incompressible Euler equations traces its origins to mid-nineteenth  century mathematical fluid dynamics. However, in fact, the number of  explicit solutions to the 2D incompressible Euler equations in Lagrangian variables   is quite limited:
Kirchhoff's elliptical vortex (1876) \cite{R5}, Gerstner's gravity wave (1809, rediscovered by Rankine 1863) \cite{R6,R7}, Ptolemaic vortices (1984) \cite{R8}, and recently discovered flows \cite{R9}.	
	
	The construction of these classical flows relies fundamentally on harmonic maps, as each admits a labeling through harmonic functions. A. Aleman and A. Constantin \cite{r1} developed a complex-analytic framework to classify such flows systematically. To extend their work, a novel approach based on harmonic mapping theory was introduced in \cite{r19}, where the authors explicitly derived all solutions—with the prescribed structural property—to the 2D incompressible Euler equations in Lagrangian variables.  However, the methods in references \cite{r1,r19} are invalid for the  2D  Boussinesq equations \eqref{A1}, because it mainly has the following  difficulties:\\
	
\noindent{ $\bullet~$} \begin{minipage}[t]{0.9\linewidth}
 The  Eq.\eqref{4.1}, that is 
\begin{align*}
  		f_t(t,z)\overline{f(t,z)}-\overline{g_t(t,z)}g(t,z)=
  		i\mathcal{K}(t,z,\overline{z}),
  	\end{align*} 	
For 2D incompressible Euler equations \eqref{A2}, \(\mathcal{K}_t=0\), while for 2D  Boussinesq equations \eqref{A1}, 	
\(\mathcal{K}_t=\mathcal{J} \big(\theta_x+\mu(\Delta v)_x-\mu(\Delta u)_y\big)\).   It is difficult to handle $\mathcal{K}(t,z,\bar{z})$.
\end{minipage}\\

 \noindent{ $\bullet~$}
 \begin{minipage}[t]{0.9\linewidth}
  It is very difficult to obtain the explicit solutions of $P$ and $\theta$, because the $K_t$ is very complicated. 
  \end{minipage}\\

\noindent{ $\bullet~$}
 \begin{minipage}[t]{0.9\linewidth}
 The 2D Boussinesq equations admit thermal diffusion term \(\kappa\Delta\theta\) and viscosity term \(\mu\Delta u,\mu\Delta v\), The difficulty faced is the need to solve the nonlinear differential equations in $\mathbb{C}^4.$
    \end{minipage}\\

In this paper, we propose a method for solving all possible solutions of 2D Boussinesq equations based on harmonic  maps.
The method overcomes the difficulty of directly solving nonlinear differential equations in $\mathbb{C}^4$.	
Our research methods are as follows:\\

\noindent{ $\bullet~$}
 \begin{minipage}[t]{0.9\linewidth}
 We introduce the Schwarzian derivative and pre-Schwarzian derivative theories defined in \cite{r7},  and study the fundamental properties of these two types of derivatives in  any locally univalent harmonic mappings. Moreover, we have fully characterized the characteristic properties of the sense-preserving harmonic mappings with equal Schwarzian derivatives and pre-Schwarzian derivatives.
  \end{minipage}\\
 
 \noindent{ $\bullet~$}
 \begin{minipage}[t]{0.9\linewidth}	
We  discover a key property, that is, when the mass conservation equation is expressed in terms of Lagrangian variables, the Schwarzian derivative and pre-Schwarzian derivative of its planar harmonic mapping exhibit time-independent characteristics. Therefore, we have derived the explicit analytical expression of the sense-preserving harmonic mappings with equal Schwarzian derivatives and pre-Schwarzian derivatives (see Theorem \ref{lem3.1} and Theorem \ref{lem2}).
\end{minipage}\\

\noindent{ $\bullet~$}
 \begin{minipage}[t]{0.9\linewidth}	
Taking advantage of the inherent characteristics of the equation, we transform the original problem into a direct solution of the analytical form of the correlation coefficient, thereby avoiding the complexity of directly dealing with nonlinear differential equation systems in high-dimensional complex spaces.	
\end{minipage}\\
	
According to the above mentioned method,  we construct and classify all possible solutions with the specified structural property, to the 2D  Boussinesq equations (in Lagrangian variables).  The method   can not only handle viscous flows with vorticity (with only the flow divergence required to be zero), but also handle ideal Euler flows. 	For instance,  for ideal Euler flows, that is  when $\mu=0$ and $\theta=0$, our results include  (3.14) and (3.15) in Ref.\cite{r1}, as well as Theorem 3 and  Theorem 4 in Ref.\cite{r19}. Our study not only obtains explicit solutions to the Boussinesq equations and reveals the profound intrinsic connections between complex analysis and fluid mechanics, but more importantly, provides novel approaches and methodologies for interdisciplinary research between mathematics and fluid mechanics.

	The rest of the paper is as follows. In Section \ref{sect2}, we derive  the governing equations in Lagrangian coordinates. In Section \ref{sect3},  we introduce   the harmonic labelling maps and the Schwarzian derivatives.  We establish the properties of sense-preserving harmonic mappings whose pre-Schwarzian and Schwarzian derivatives coincide. The first class of solutions exhibiting enhanced structural simplicity is systematically constructed in Section \ref{sect4}, whereas the approach generating the general solution families characterized by greater geometric complexity is rigorously established in  Section \ref{sect5}.

\section{The Governing Equations}\label{sect2}
	In this section, we derive the governing equations in Lagrangian coordinates.
The  2D  Boussinesq equations consists of the momentum equations
	  \begin{equation}\label{1.1}
	 	\begin{cases}
	 		u_t+uu_x+vu_y+P_x=\mu\Delta u,\\
	 		v_t+uv_x+vv_y+P_y=\mu\Delta v+\theta,\\
	 	\end{cases}
	 \end{equation}	 
and temperature  equation
\begin{align}\label{B1}
	 \theta_t+u \theta_x+v\theta_y=\kappa\Delta\theta,
	 	 \end{align}
	 and the mass conservation equation
	 \begin{align}\label{1.2}
	 	u_x+v_y=0.
	 \end{align}
	 Eqs.\eqref{1.1}-\eqref{B1} is equivalent to the following system 
	 \begin{align}\label{p1}
	 	\begin{cases}
	 		\omega_t+u \omega_x+v\omega_y=\mu\Delta\omega+\theta_x,\\
	 			\theta_t+u \theta_x+v\theta_y=\kappa\Delta\theta,
					 	\end{cases}
	 \end{align}
	 where \(\omega\) represents the vorticity of the fluid, given by
	 \begin{align}
	 	\omega=v_x-u_y.
	 \end{align} It is well-known that
	 the Eulerian description identifies  the motion of the fluid entirely in terms of the velocity field $(u(t, x, y), v(t, x, y))$ in space $(x, y)$ and time $t$.
	 Then the Lagrangian coordinates  provides the most complete representation of the flow in which the motion of all fluid particles is described. 
	 If the velocity field $\left(u(t,x,y),v(t,x,y)\right)$ is known, the motion of the individual particles $\left(x(t),y(t)\right)$ is obtained by integrating a system of ordinary differential equations
	  \begin{equation}
	 	\begin{cases}
	 		x'=u(t,x,y),\\
	 		y'=v(t,x,y),
	 	\end{cases}
	 \end{equation}
	 whereas the knowledge of the particle path $t\mapsto (x(t), y(t))$
	 provides by differentiation with respect to $t$ the velocity field at time $t$ and at the location $(x(t),y(t))$. 
	In the Lagrangian framework, the (now dependent) variables $x$ and $y$ denote the position of a particle at time $t$ and 
are functional of a label (a Lagrangian coordinate). While it is possible to use the particle's initial position at $t=0$  to label a particle (for example, when describing particle trajectories beneath a water wave \cite{r14,r15,R10}), but this method is  inconvenient because it is fundamentally tied to the initial configuration of the fluid domain, potentially limiting its utility for dynamic particle tracking in evolving systems.    
	
	We introduce complex Cartesian coordinates \(x+iy\) and complex
	 Lagrangian coordinates \(a + ib\). We can take a simply connected domain $\Omega_0$ to represent the labelling initial domain. Considering the injective map	 	 
     \begin{align}\label{1.4}
	    (a,b)\mapsto (x(t;a,b),y(t;a,b)),
	 \end{align}
	 then by the label $(a,b)\in \Omega_0$ we can identify the evolution in time of a specific particle, the fluid domain $\Omega(t)$, being the image of $\Omega_0$ under the map of  $\eqref{1.4}$.  To attain the governing equations in Lagrangian coordinates, we consider the following coordinate transformation:
	   \begin{align}\label{1.5}
	 	\begin{cases}
	 	\displaystyle	u(t,x,y)=\frac{\partial}{\partial t}x(t;a,b),
	 	\vspace{10pt}\\
	 	\displaystyle v(t,x,y)=\frac{\partial}{\partial t}y(t;a,b),
	 	\end{cases}
	 \end{align}
	 and the relations
	 \begin{align}
	 	\begin{cases}
	 		\displaystyle  \frac{\partial}{\partial a}= x_a\frac{\partial}{\partial x}+y_a\frac{\partial}{\partial y},\vspace{10pt}\\
	 		\displaystyle  \frac{\partial}{\partial b}= x_b\frac{\partial}{\partial x}+y_b\frac{\partial}{\partial y},
	 	\end{cases}
	 	 \end{align}
then we obtain
	  \begin{align}\label{1.7}
	 	\begin{cases}
	 	\displaystyle	\frac{\partial}{\partial x}=\frac{1}{\mathcal J}(y_b\frac{\partial}{\partial a}-y_a\frac{\partial}{\partial b}),\vspace{10pt}\\
	 	\displaystyle\frac{\partial}{\partial y}=\frac{1}{\mathcal J}(x_a\frac{\partial}{\partial b}-x_b\frac{\partial}{\partial a}).
	 	\end{cases}
	 \end{align}
	 Here $\mathcal{J}$ denotes the Jacobian of the transformation \eqref{1.4}, given by
	 \begin{align}\label{1.8}
	 	\mathcal{J}=\left|\frac{\partial (x,y)}{\partial(a,b)}\right|=x_ay_b-y_ax_b\neq0,
	 \end{align}
	 which ensures the local injectivity of the map \eqref{1.4}.
	 In light of $\eqref{1.5}$-$\eqref{1.7}$, Eq.$\eqref{1.2}$ becomes
	 \begin{align*}
	 	u_x+v_y=\frac{x_ay_{bt}-x_by_{at}+y_bx_{at}-y_ax_{bt}}{\mathcal{J}}=\frac{\mathcal J_t}{\mathcal{J}}=0.
	 \end{align*}
	 Since $\mathcal{J}\neq0$, we get
	 \begin{align}\label{1.9}
	 	\mathcal J_t=0,
	 \end{align}
	 that is
	  \begin{align}\label{Q}
	  	\left(x_ay_b-y_ax_b\right)_t=0.
	\end{align}
	In other words, a small set of particles \(\ud a\times \ud b\) must always enclose the same physical area \(\mathcal{J}^{-1} \ud x \times\ud y\) over time, otherwise, the flow would permit compression.
	 To get the scalar function $\theta(t,x,y)$ in Lagrangian coordinates, we introduce the  mapping
	 \begin{align*}
	 	(a,b)\mapsto \Gamma(t;a,b),
	 \end{align*}
	 where 
	 \begin{align}\label{1.10}
	 	\frac{\partial}{\partial t}\Gamma(t;a,b)=\theta(t,x,y).
	 \end{align}
	 Moreover, differentiating $\eqref{1.5}$ and $\eqref{1.10}$ with respect to $t$, it is easy to find that
	 	 \begin{align*}
	 	\begin{cases}
	 		x_{tt}=u_t+uu_x+vu_y,\\
	 		y_{tt}=v_t+uv_x+vv_y,\\
	 		\Gamma_{tt}=\theta_t+u\theta_x+v\theta_y.
	 	\end{cases}
	 \end{align*}
	 Then Eqs. $\eqref{1.1}-\eqref{B1}$, in Lagrangian coordinates, take the form
	 \begin{align}\label{ab}
	 	\begin{cases}
	 		x_{tt}=-\frac{1}{\mathcal{J}}(P_ay_b-P_by_a)+\mu\Delta(x_t),\\
	 		y_{tt}=\Gamma_t-\frac{1}{\mathcal{J}}(P_bx_a-P_ax_b)+\mu\Delta(y_t),\\
	 		\Gamma_{tt}=\kappa\Delta(\Gamma_t).
	 	\end{cases}
	 \end{align}
	 By $\eqref{1.8}$, we find that
	  \begin{align*}
	 	\begin{cases}
	 		P_a=\mu\left(x_a\Delta(x_t)+y_a\Delta(y_t)\right)+y_a\Gamma_t-x_ax_{tt}-y_ay_{tt},\\
	 			P_b=\mu\left(x_b\Delta(x_t)+y_b\Delta(y_t)\right)+y_b\Gamma_t-x_bx_{tt}-y_by_{tt}.\\
	 	\end{cases}
	 \end{align*}
	 Then the above equations are equivalent to the requirement \( P_{ab} = P_{ba} \), which implies
	 \begin{align}\label{A}
	& x_{att}x_b+y_{att}y_b+y_a\Gamma_{bt}+\mu\left(x_a(\Delta(x_t)_b)+y_a(\Delta(y_t))_b\right)\nonumber\\
	 =&x_{btt}x_a+y_{btt}y_a+
	 y_b\Gamma_{at}+\mu\left(x_b(\Delta(x_t))_a+y_b(\Delta(y_t))_a\right).
	\end{align}
	 Now, we  calculate the vorticity \(\omega\) in Lagrangian coordinates. In light of \eqref{1.5} and \eqref{1.7}, we have
	 \begin{align*}
	 	\omega&=v_x-u_y\nonumber\\[5pt]
	 	&=\displaystyle\frac{y_{at}y_b-y_{bt}y_a+x_{at}x_b-x_{bt}x_a}{\mathcal{J}}.
	 \end{align*}
	 Furthermore, we obtain
	 \begin{align*}
	 \mathcal{J}\partial_t\omega=x_{att}x_b+y_{att}y_b-x_{btt}x_a-y_{btt}y_a.
	 \end{align*}
	   According to the above considerations, we derive that the governing equations $\eqref{1.1}$-$\eqref{1.2}$ are equivalent to $\eqref{A}$ and $\eqref{Q}$, plus the requirement that, at any time \(t\), the map \( \eqref{1.4}\) is a global diffeomorphism  from the label domain $\Omega_0$ to the fluid domain \(\Omega(t)\).

	\section{Harmonic Maps}\label{sect3}
	In this section, we introduce the harmonic labelling maps to transform the governing equations \eqref{1.1}-\eqref{1.2} into a complex differential system \eqref{4.1} in \(\mathbb{C}^4.\) Moreover, we use the new definition for the Schwarzian derivative of
	harmonic mappings, and derive the properties of the sense-preserving harmonic mappings with equal Schwarzian derivatives and Jacobians.

	\subsection{Harmonic labeling  maps}
	In this subsection, we develop an approach that determines all fluid flows where the particle labelling \eqref{1.4} in Lagrangian coordinates is expressed through a harmonic mapping at every time \(t.\)
	Since our methods mainly rely on complex analysis, it is necessary to introduce some complex analysis notation. A complex-valued function \(K\) is harmonic in a simply connected domain \(\Omega_0\subset \mathbb{C}\) if Re\((K)\) and Im\((K)\) are real harmonic in \(\Omega_0.\) Every such \(K\) has a canonical representation \(K=F+\overline{G}\) that is unique up to an additive constant, where  \(F\)
	and \(G\) are analytic in \(\Omega_0\) (see \cite{r4}).  To find solutions to  $\eqref{A}$ and $\eqref{Q}$, we make
	in $\eqref{1.4}$ the Ansatz
	\begin{equation}\label{1.15}
	x(t;a,b)+iy(t;a,b)=F(t,z)+\overline{G(t,z)},\quad z=a+ib,
	\end{equation}
    where  \(z\mapsto F(t,z)\) and \(z\mapsto G(t,z)\) 
    are analytic in the simply connected domain $\Omega_0\subset \mathbb{C}$ at every instant $t$.
    Due to the analyticity of $F$, then \(\frac{\partial F}{\partial \overline{z}}=0.\) Moreover, we have
    \begin{equation}\label{1.16}
    	\begin{cases}
    	\displaystyle  \frac{\partial }{\partial a}=\displaystyle\frac{\partial }{\partial z}+\frac{\partial }{\partial \bar{z}},\vspace{10pt}\\
    	\displaystyle  \frac{\partial }{\partial b}=i(\frac{\partial }{\partial z}-\frac{\partial }{\partial \bar{z}}),\\
    	\end{cases}
    \end{equation}
    and\[\frac{\partial\overline{\mathcal F} }{\partial \bar{z}}=\overline{\left(\frac{\partial \mathcal F }{\partial z}\right)}.\]
    From the harmonic map $\eqref{1.15}$ together with $\eqref{1.16}$, we obtain 
        \begin{equation}\label{1.17}
    	\begin{cases}
    		x_a+iy_a=F'+\overline{G'},\\
    		x_a-iy_a=\overline{F'}+G',\\
    		x_b+iy_b=i(F'-\overline{G'}),\\
    		x_b-iy_b=i(G'-\overline{F'}).
    	\end{cases}
    \end{equation}
    Note that
  \begin{align*}
  	\mathcal{J}
  	&=x_ay_b-x_by_a\\
  	&=\text{Im}\big((x_a-iy_a)(x_b+iy_b)\big)\\
  	&=\text{Im}\left(i(\overline{F'}+G')\left(F'-\overline{G'}\right)\right)\\
  	&=|F'|^2-|G'|^2,
  \end{align*}
 together with $\eqref{1.9}$, we have
  \begin{align}\label{a1}
  \left(|F'|^2-|G'|^2\right)_t=0. 
  \end{align}
  Furthermore, from $\eqref{1.17}$ we get
  \begin{align}\label{a}
  	F'_t\overline{F'}-\overline{G'_t}G'-F'_tG'+\overline{F'}\overline{G'_t}
  	&=\left(F'+\overline{G'}\right)_t(\overline{F'}-G')\nonumber\\
  	&=i(x_{at}+iy_{at})(x_b-iy_b)\nonumber\\
  	&=(x_{at}+iy_{at})(y_b+ix_b)\nonumber\\
  	&=x_{at}y_b-x_by_{at}+i(x_{at}x_b+y_{at}y_b).
  	  \end{align}
  	Similarly, we deduce that
  	  \begin{align}\label{b}
  		F'_t\overline{F'}-\overline{G'_t}G'+F'_tG'-\overline{F'}\overline{G'_t}
  		&=\left(F'-\overline{G'}\right)_t(\overline{F'}+G')\nonumber\\
  		&=-i(x_{bt}+iy_{at})(x_a-iy_a)\nonumber\\
  		&=-(x_{bt}+iy_{bt})(y_a+ix_a)\nonumber\\
  		&=x_{a}y_{bt}-x_{bt}y_{a}-i(x_{a}x_{bt}+y_{a}y_{bt}).
  	\end{align}
  	Adding $\eqref{a}$ to $\eqref{b}$, we get
  	\begin{align}\label{c1}
  			\text{Re}\left(F'_t\overline{F'}-\overline{G'_t}G'\right)
  			&=\frac{1}{2}\left(x_ay_{bt}-x_by_{at}+y_bx_{at}-y_ax_{bt}\right)\nonumber\\
  			&=\frac{\mathcal J_t}{2}\nonumber\\
  			&=0.
  	\end{align}
  	Similarly,  differentiating  $\eqref{a}$ and $\eqref{b}$ with respect to $t$ together with $\eqref{Q}$, we have
  	  	\begin{align}\label{c2}
  		\left\{\text{Im}(F'_t\overline{F'}-\overline{G'_t}G')\right\}_t
  		&=\frac{1}{2}\left( x_{att}x_b+y_{att}y_b-x_{btt}x_a-y_{btt}y_a\right)\nonumber\\
  		&=\frac{1}{2}(y_b\Gamma_{at}-y_a\Gamma_{bt})-\frac{\mu}{2}\big(x_a(\Delta(x_t)_b)-x_b(\Delta(x_t))_a\big)\nonumber\\
  		&\quad+\frac{\mu}{2}\big(y_b(\Delta(y_t)_a)-y_a(\Delta(x_t))_b\big).
  	\end{align}
  	From the  relations \eqref{c1}-\eqref{c2}, we find that
  	\begin{align}\label{aa}
\notag  	&\left(F'_t(t,z)\overline{F'(t,z)}-\overline{G'_t(t,z)}G'(t,z)\right)_t\\
\notag	=&\frac{i}{2}(y_b\Gamma_{at}-y_a\Gamma_{bt})-\frac{i\mu}{2}\big(x_a(\Delta(x_t)_b)-x_b(\Delta(x_t))_a\big)\\&+\frac{i\mu}{2}\big(y_b(\Delta(y_t)_a)-y_a(\Delta(x_t))_b\big).
  \end{align} 
  Let 
  \[F'=f,\quad G'=g.\] 
  Define the map \(\mathcal {L}: C^1([0,\infty);\Omega_0)\mapsto C^1([0,\infty);\mathbb{C}),\)
  and
  \[\mathcal {L} f=f_t\overline{f},\]
  which implies
  \begin{align}\label{R}
  	&\mathcal {L} \overline{f}=\overline{f_t}f=\overline{\mathcal {L} f},\nonumber\\
  	&\mathcal {L}(\lambda f)=|\lambda|^2\mathcal {L}f,~\lambda\in \mathbb{C},\nonumber\\
  	&\mathcal {L} (f+ g)=\mathcal {L}f+\mathcal {L}g+\overline{f}g_t+f_t\overline{g},\nonumber\\
  	&\mathcal {L} (f\cdot g)=|f|^2\mathcal {L}g+|g|^2\mathcal {L}f,\nonumber\\
  	&\mathcal {L}\left(\frac{f}{g}\right)=\frac{|g|^2\mathcal {L}f-|f|^2\mathcal {L}g}{g^2}, g\neq0.  
  \end{align}
  Integrating \eqref{aa} from 0 to $t$ yields that
  	  	\begin{align}\label{4.1}
  		f_t(t,z)\overline{f(t,z)}-\overline{g_t(t,z)}g(t,z)=
  		i\mathcal{K}(t,z,\overline{z}),
  	\end{align} 
  	where 
  	\begin{align}\label{tt1}
  		\mathcal{K}=\frac{1}{2}\int_{0}^{t}\mathcal{J} \big(\theta_x+\mu(\Delta v)_x-\mu(\Delta u)_y\big) \ud s+\ell(z,\bar{z}),
  	\end{align}
  	and 
  	\begin{align*}
  		\begin{cases}
  			\displaystyle \frac{\partial}{\partial x}=\displaystyle \frac{1}{\mathcal{J}}\left[(\overline{f}-\overline{g})\frac{\partial}{\partial z}+(f-g)\frac{\partial}{\partial {\bar{z}}}\right],\\[14pt]
  				\displaystyle \frac{\partial}{\partial y}=\displaystyle \frac{i}{\mathcal{J}}\left[(\overline{f}+\overline{g})\frac{\partial}{\partial z}-(f+g)\frac{\partial}{\partial {\bar{z}}}\right].
  		\end{cases}
  	\end{align*}
  	
  Next, we will derive the explicit forms of pressure \(P\) and temperature field \(\theta\)  through Eq.\eqref{4.1}.
\begin{theorem}
		The  temperature field  is
		\begin{align*}
			\theta=\theta_0+\mathcal{R}+\int_{0}^{t}(\kappa\Delta-D_t)\mathcal{R}~\ud s,
		\end{align*}
		where  \(\theta_0=\theta(0,\cdot),\)
		\begin{align*}
			\mathcal{R}=\int_{0}^{x}\left(\frac{2\mathcal{K}_t}{\mathcal{J}}+\mu(\Delta u)_y-\mu(\Delta v)_X\right)\ud X
		\end{align*}
		and
		\begin{align*}
		D_t=\partial_t+(u,v)\cdot\nabla.
		\end{align*}
		Moreover, the pressure is
		\begin{align}\label{oo}
			P=\mathfrak{P}+\int_{0}^{x}(\mu\Delta-D_t)u~\ud X+\int_{0}^{y}\left((\mu\Delta-D_t)v+\theta\right)~\ud Y,
		\end{align}
		where \(\mathfrak{P}\) depends only on \(t.\)
\end{theorem}
\begin{proof}
	 From \eqref{tt1}, then
	 \begin{align*}
	\theta=\mathfrak{T}+\int_{0}^{x}\left(\frac{2\mathcal{K}_t}{\mathcal{J}}+\mu(\Delta u)_y-\mu(\Delta v)_X\right)\ud X
	\end{align*}
	for some function \(\mathfrak{T}.\)
	 Define
	 \begin{align*}
	 \frac{D}{D t}=\frac{\partial}{\partial t}+ (u,v)\cdot \nabla.
	 \end{align*} 
	 By \eqref{1.1}, then 
	 \begin{align*}
	 (D_t-\kappa\Delta)\theta=0,
	 \end{align*}
	thus
	\begin{align*}
	 \partial_t\mathfrak{T}+v\mathfrak{T}_y-\kappa\mathfrak{T}_{yy}=(\kappa\Delta-D_t)\int_{0}^{x}\left(\frac{2\mathcal{K}_t}{\mathcal{J}}+\mu(\Delta u)_y-\mu(\Delta v)_X\right)\ud X.
	 \end{align*}
	 Note that $\mathfrak{T}$ is independent of \(x,\) then we have
	 \begin{align*}
	 	\mathfrak{T}_x&=\frac{1}{\mathcal{J}}\left((\overline{f}-\overline{g})\mathfrak{T}_z+(f-g)\mathfrak{T}_{\bar{z}}\right)\\
	 	&=0.
	 \end{align*}
	 since 
	 \[\mathcal{J}=|f|^2-|g|^2\neq0,\] 
	 then
	\(\mathfrak{T}\)  depends only  on \(t.\) Therefore
	\begin{align*}
	\mathfrak{T}=\theta_0+\int_{0}^{t}\left((\kappa\Delta-D_t)\int_{0}^{x}\left(\frac{2\mathcal{K}_t}{\mathcal{J}}+\mu(\Delta u)_y-\mu(\Delta v)_X\right)\ud X\right)\ud s.
\end{align*}	
	Using  Eq.\eqref{1.1} again, then we get
	\begin{align*}	
		P(t,x,y)=\mathfrak{A}(t,x)+\mathfrak{B}(t,y)+\int_{0}^{x}(\mu\Delta-D_t)u~\ud X+\int_{0}^{y}\left((\mu\Delta-D_t)v+\theta\right)~\ud Y.
\end{align*}		
		As we mentioned before,  \(\mathfrak{A}_y=0\) and \(\mathfrak{B}_x=0\) lead to
		\begin{align*}
			\mathfrak{A}=\mathfrak{A}(t),\quad\mathfrak{B}=\mathfrak{B}(t).
		\end{align*}
		Denote \(\mathfrak{P}=\mathfrak{A}+\mathfrak{B},\) we obtain \eqref{oo}.
\end{proof}

\begin{remark}
	By \eqref{1.15}, then 
	\begin{align*}
		u=\textup{Re}\left\{\int_{0}^{z}\left(f_t+\overline{g}_t\right)\ud w\right\}
	\end{align*}
	and
		\begin{align*}
		v=\textup{Im}\left\{\int_{0}^{z}\left(f_t+\overline{g}_t\right)\ud w\right\}.
	\end{align*} Hence, in order to find the solutions (\(u, v, P, \theta\)) to the governing equations \eqref{1.1}-\eqref{1.2}, it suffices to obtain the solutions (\(f,g\)) to Eq.\eqref{4.1}. Moreover, in incompressible flow, the pressure  adjusts dynamically to maintain zero divergence in the velocity field. From  \eqref{oo}, we can see that  the pressure  is determined by the velocity and temperature fields of the fluid.
\end{remark}
The harmonic mapping \(K=F+\overline{G}\) is locally univalent if and only if its Jacobian \(\mathcal{J}\) does not vanish in \(\Omega_0\) (see \cite{r3}). It is known that a locally
    univalent harmonic mapping \(K\) is sense-preserving if its Jacobian 
    is positive and sense-reversing if \(\mathcal{J}<0\). If \(K\) is sense-preserving, then \(\overline{K}\) is sense-reserving and its Jacobian \(\mathcal{J}_1\) satisfies \(\mathcal{J}_1=|G'|^2-|F'|^2<0.\)  Moreover, the  dilatation \(q=G'/F'
  \) of the
    harmonic mapping \(K=F+\overline{G}\) is analytic in \(\Omega_0\). If  the harmonic mapping \(K\) is not a constant, then it is sense-preserving if
    and only if \(|q|\leq 1.\)  For a detailed
    discussion of univalence criteria  on  harmonic
    maps  we refer the reader to \cite{r8,r9,r10,r22}. Set \(F_0:=F(0,\cdot), G_0:=G(0,\cdot), f_0=F'_0,\) and \(g_0=G'_0.\) Without loss of generality we assume  the map \(z\mapsto F_0(z)+\overline{G_0(z)}\) is sense-preserving in the simply connected domain \(\Omega_0\). So we obtain that \(\mathcal{J}=|f_0|^2-|g_0|^2>0,\) implies \(|f_0|>0\) so that 
    the  analytic dilatation \(q=g_0/f_0\) satisfies \(|q|<1.\)

In order to find the solutions $f\not= 0$ and $g\not= 0$ such that the governing equation \eqref{4.1} holds, we need to prove the following theorem.

      	\begin{theorem}\label{thm4}
    	Let \(\Omega_0\subset \mathbb{C}\) be a simply connected domain. Assume that the initial harmonic labelling mapping \(F_0+\overline{G_0}\) is  sense-preserving in \(\Omega_0\). Then
    	\begin{align*}
    		i\mathcal{K}(t,z,\overline{z})=
    		\mathcal{C}_1(t)|f_0(z)|^2+\mathcal{C}_2(t)|g_0(z)|^2+\mathcal{C}_3(t)f_0(z)\overline{g_0(z)}+\mathcal{C}_4(t)\overline{f_0(z)}g_0(z),
    	\end{align*} 
    	where \(\mathcal{C}_1,\mathcal{C}_2,\mathcal{C}_3,\mathcal{C}_4: [0,\infty)\mapsto \mathbb{C}\) are \(C^1\) functions.
    \end{theorem}
    \begin{proof}
    	Since the Jacobian of the
    	labelling map \(\eqref{1.15}\) remains unchanged at all times \(t,\)  we can deduce that 
    	\begin{align}\label{4.2}
    		|f(t,z)|^2-|g(t,z)|^2=|f_0(z)|^2-|g_0(z)|^2.
    	\end{align}
    	Since we seek the solutions \(f\neq0\) and \(g\neq0,\) we set \(q_1(z)=g_0(z)/f_0(z), \) \(m_1(t,z)=f(t,z)/f_0(z)\) and \(m_2(t,z)=g(t,z)/f_0(z),\) so that  \eqref{4.2} becomes
    	\begin{align}\label{4.3}
    		|m_1(t,z)|^2-|m_2(t,z)|^2=1-|q_1(z)|^2.
    	\end{align}
    	Moreover, applying the operator \(\Delta=4\partial_z\partial_{\bar{z}}\) to  \eqref{4.3}, we obtain
    	\begin{align}\label{4.0}
    		|\partial_zm_1(t,z)|^2=|\partial_zm_2(t,z)|^2-|q'_1(z)|^2.
    	\end{align}
    	Since we assume that the initial harmonic labelling mapping \(F_0+\overline{G_0}\) is  sense-preserving, then the dilatation \(q_1(z)\) of \(F_0+\overline{G_0}\) is analytic in \(\Omega_0,\) and satisfies \(0<|q_1(z)|<1.\) Due to the analyticity of \(f\) and \(g,\) then \(m_1\) and \(m_2\) are also analytic. We take logarithms in \(\eqref{4.3}\) to get
    	\begin{align}\label{4.4}
    		\log m_1(t,z)+\log \overline{m_1(t,z)}=\log\left(|m_2(t,z)|^2+1-|q_1(z)|^2\right).
    	\end{align}
    	Applying the operator \(\Delta=4\partial_z\partial_{\bar{z}}\) to  \eqref{4.4}, we find that \(\Delta\log|m_1(t,z)|^2=0.\)
    	So the function of the left-hand side of \eqref{4.4} is harmonic so that the one on the
    	right-hand side must be harmonic as well. Furthermore, we have 
    	\begin{align}\label{4.5}
    		\Delta \log\left(|m_2(t,z)|^2+1-|q_1(z)|^2\right)
    		=0.
    	\end{align}
    	$\mathbf{Case ~1}\quad q'_1(z)=0.$ If \(q_1'(z)=0\) then we get \(q_1\) equals a  constant \(c_0\in \mathbb{C}\setminus\{0\}.\) By   \eqref{4.5}, we further obtain 
    	\begin{align*}
    		\Delta \log\left(|m_2(t,z)|^2+1-|c_0|^2\right)&=\frac{4|\partial_zm_2(t,z)|^2}{\left(|m_2(t,z)|^2+1-|c_0|^2\right)^2}\\
    		&=0,
    	\end{align*}
    	which implies that \(m_2(t,z)=\rho_2(t).\) From \eqref{4.0}, then \(m_1(t,z)=\rho_1(t).\) Here \(\rho_1,\rho_2\) are \(C^1\)  complex functions.  Then we find 
    	\begin{align*}
    		f_t(t,z)\overline{f(t,z)}-\overline{g_t(t,z)}g(t,z)&=\left(\rho_1'(t)\overline{\rho_1(t)}-\overline{\rho_2'(t)}\rho_2(t)\right)|f_0(z)|^2\\
    		&=\mathcal{C}_0(t)|f_0(z)|^2.
    	\end{align*} 
    	$\mathbf{Case ~2}\quad q_1'(z)\neq0.$ Let 
    	\[q_2(t,z)=\frac{\partial_zm_1(t,z)}{q'_1(z)} \quad \text{and}\quad q_3(t,z)=\frac{\partial_zm_2(t,z)}{q'_1(z)} .\] Then the \eqref{4.0} becomes 
    	\begin{align}\label{4.6}
    		|q_3(t,z)|^2=1+|q_2(t,z)|^2.
    	\end{align}
    	Similarly, taking logarithms in \(\eqref{4.6}\) yields that
    	\begin{align}\label{4.7}
    		\log q_3(t,z)+\log \overline{q_3(t,z)}=\log\left(1+|q_2(t,z)|^2\right).
    	\end{align}
    	Notice that at fixed instant \(t\geq0\) the map \(z\mapsto q_3(t,z)\) is analytic. As before, applying the operator \(\Delta=4\partial_z\partial_{\bar{z}}\) to \eqref{4.7}, we have \[ \Delta\left(\log q_3(t,z)+\log\overline{q_3(t,z)}\right)=0.\] Furthermore, in light of \eqref{4.6}, we can deduce that 
    	\begin{align*}
    		\Delta \log\left(1+|q_2(t,z)|^2\right)&=\frac{4|\partial_zq_2(t,z)|^2}{\left(1+|q_2(t,z)|^2\right)^2}\\
    		&=0.
    	\end{align*}
    	There exist two \(C^1\) complex function \(\rho_3(t),m_0(t)\) such that \[m_1(t,z)=\rho_3(t)q_1(z)+m_0(t).\] Moreover, by \eqref{4.6}, we obtain that \[m_2(t,z)=\sqrt{1+|\rho_3(t)|^2}\mathrm{e}^{i\rho_4(t)}q_1(z)+m^*(t),\]
    	for two \(C^1\) complex functions \(\rho_4(t),m^*(t).\) Since
    	\begin{align*}
    		f_t(t,z)\overline{f(t,z)}-\overline{g_t(t,z)}g(t,z)=\left(\overline{m_1(t,z)}\partial_tm_1(t,z)-m_2(t,z)\overline{\partial_tm_2(t,z)}\right)|f_0(z)|^2,
    	\end{align*}
    	a straightforward calculation shows that
    	\begin{align*}
    		i\mathcal{K}(t,z,\overline{z})&=f_t(t,z)\overline{f(t,z)}-\overline{g_t(t,z)}g(t,z)\\
    		&=\mathcal{C}_1(t)|f_0(z)|^2+\mathcal{C}_2(t)|g_0(z)|^2\\[5pt]
    		&\quad+\mathcal{C}_3(t)f_0(z)\overline{g_0(z)}+\mathcal{C}_4(t)\overline{f_0(z)}g_0(z).
    	\end{align*}
    \end{proof}

\subsection{The Schwarzian and pre-Schwarzian derivatives}
In this subsection, we derive several properties related to the pre-Schwarzian and Schwarzian  derivatives for locally univalent harmonic mappings in a simply connected domain.
The \textit{ Schwarzian derivative} \(S_H\) of a locally univalent harmonic function \(K\) with Jacobian \(\mathcal{J}\) was defined in \cite{r7} by
    \begin{align*}
    	S_H(K)=\frac{\partial}{\partial z}\left(P_H(K)\right)-\frac{1}{2}\left(P_H(K)\right)^2,
    \end{align*}
    where \(P_H(K)\) is the \textit{pre-Schwarzian derivative} of \(K\), which
    equals
    \begin{align*}
    P_H(K)=\frac{\partial}{\partial z}\log \mathcal{J}=\frac{F''}{F'}-\frac{q'\overline{q}}{1-|q|^2}.
    \end{align*}
 It is not difficult to find that \(S_H(K)=S_H(\overline{K}).\)    Hence, without loss of generality we assume that  \(K\) is sense-preserving in \(\Omega_0.\)
 The  Schwarzian derivative of sense-preserving harmonic
mapping \(K=F+\overline{G}\) with dilatation \(q=G'/F'
\) can be written as 
\begin{align*}
S_H(K)=S(F)+\frac{\overline{q}}{1-|q|^2}\left(\frac{F''}{F'}q'-q''\right)-\frac{3}{2}\left(\frac{q'\overline{q}}{1-|q|^2}\right)^2,
\end{align*}
where \(S(F)\) is the classical Schwarzian derivative, which is defined by
\[S(F)=\left(\frac{F''}{F'}\right)'-\frac{1}{2}\left(\frac{F''}{F'}\right)^2.\]

We begin by proving the following key theorem, essential for deriving our main results. The next theorem characterizs the sense-preserving harmonic mappings with equal Schwarzian derivatives and Jacobians .
\begin{theorem}\label{lem3.1}
	Let \(K_1=F_1+\overline{G_1}\) and \(K_2=F_2+\overline{G_2}\)  be two sense-preserving harmonic mappings in a simply connected domain \(\Omega_0\subset \mathbb{C}\) with  dilatations \( p_1=G'_1/F'_1 \) and \( p_2=G'_2/F'_2 \). Set \(\mathcal{J}_1,\mathcal{J}_2\) be the  Jacobians of the harmonic mappings \(K_1\) and \( K_2\), respectively. There are the following properties:
	\begin{enumerate}
		\item[\textup{(i)}] \(S_H(K_1)\) is analytic if and only if \(p_1\) is a constant;
		
		\item[\textup{(ii)}] If \(G'_1=\lambda F'_1\), where \(\lambda\in \mathbb{C},\) and \(\mathcal{J}_1=\mathcal{J}_2,\) that is
		\begin{align*}
		|F'_1|^2-|G'_1|^2=|F'_2|^2-|G'_2|^2,
		\end{align*}
		then there are two constants \(\alpha,\beta \in\mathbb{C}\) such that \(F'_2=\alpha F'_1\) and \(G'_2=\beta F'_1\), where \(|\alpha|^2-|\beta|^2=1-|\lambda|^2>0\);
		\item[\textup{(iii)}]If \(G'_1=\lambda F'_1\), where \(\lambda\in \mathbb{C},\) and \(S_H(K_1)=S_H(K_2),\) then  \(F'_2=(\mathcal{T}\circ F_1)'\) and \(G'_2=c(\mathcal{T}\circ F_1)',\) where \(c\in \mathbb{C} \) with \(|c|<1\) is a constant,
		and \(\mathcal T\) is non-constant M\"{o}bius transformation of the form 
		\begin{align*}
		\mathcal{T}(z)=\frac{mz+n}{sz+d},\quad z\in \mathbb{C}, \quad md-ns\neq0.
		\end{align*}
	\end{enumerate}

\end{theorem}
\begin{proof}
(i) If the 	dilatation \(p_1\) of \(K_1\) is a constant, then \(S_H(K_1)=S(F).\)  Since \(F\) is analytic in \(\Omega_0,\)  then \( S_H(K_1)\) is also analytic. Suppose that a sense-preserving harmonic mapping \( K_1=F_1+\overline{G}_1\) with dilatation \( p_1=G'_1/F'_1\) has analytic Schwarzian derivative \(S_H(K_1)\) defined by
 \begin{align}\label{l1}
 	S_H(K_1)=S(F_1)+\frac{\overline{p_1}}{1-|p_1|^2}\left(\frac{F''}{F'}p_1'-p_1''\right)-\frac{3}{2}\left(\frac{p_1'\overline{p_1}}{1-|p_1|^2}\right)^2.
 \end{align}
  Assume that \(p_1\) is not a constant.
 Multiplying \eqref{l1} by \(\left(1-|p_1|^2\right)^2\) yields that
 \begin{align}
 (S_H(K_1)-S(F_1))\left(1-|p_1|^2\right)^2+\overline{p_1}\left(1-|p_1|^2\right)^2\left(\frac{F''}{F'}p_1'-p_1''\right)-\frac{3}{2}
 p_1'^2\overline{p_1}^2=0.
 \end{align}
 Let 
 \begin{align*}
 E=S_H(K_1)-S(F_1),
 \end{align*}
 then we obtain
 \begin{align}\label{l2}
 	E+\overline{p_1}\left(\frac{F''}{F'}p_1'-p_1''-2p_1E\right)+\overline{p_1}^2\left(p_1^2E-p_1\left(\frac{F''}{F'}p_1'-p_1''\right)-\frac{3}{2}p_1'^2\right)=0.
 \end{align}
Differentiating \(\eqref{l2}\) with respect to \(\bar{z}\) yields that 
\begin{align}\label{l3}
	\overline{p_1'}\left(\frac{F''}{F'}p_1'-p_1''-2p_1E\right)+2\overline{p_1p_1'}\left(p_1^2E-p_1\left(\frac{F''}{F'}p_1'-p_1''\right)-\frac{3}{2}p_1'^2\right)=0.
\end{align}
Since we assume that \(p_1\) is not a constant, then there is an open disk \(D(z_0,\delta_0)\subset \Omega_0\) with center \(z_0\) and radius \(\delta_0>0\) where \(p_1'\neq0.\) So we can divide \eqref{l3} by \(\overline{p_1'},\) and take derivatives with respect to \(\bar{z}\) to get
\begin{align}\label{l4}
	\overline{p_1'}\left(p_1^2E-p_1\left(\frac{F''}{F'}p_1'-p_1''\right)-\frac{3}{2}p_1'^2\right)=0.
\end{align}
Since \(p_1'\neq0\) in \(D(z_0,\delta_0),\) we have
\begin{align}\label{l5}
p_1^2E-p_1\left(\frac{F''}{F'}p_1'-p_1''\right)-\frac{3}{2}p_1'^2=0.
\end{align}
In light of \eqref{l3} and \eqref{l5}, we can know that
\begin{align}\label{l6}
\frac{F''}{F'}p_1'-p_1''-2p_1E=0.
\end{align}
From \eqref{l5} and \eqref{l6} we have \(E=0.\)
However, from \eqref{l6} we deduce
\begin{align*}
\frac{F''}{F'}p_1'-p_1''=0.
\end{align*}
Finally, using \eqref{l5}, we get \(p_1'=0\) in \(D(z_0,\delta_0),\)
which contradicts with our assumption. Therefore  the Schwarzian derivative \(S_H(K_1)\) is analytic in \(\Omega_0, \) such that the dilation \(p_1\) is a constant.\\
\indent (ii) According to the above considerations, \(G'_1=\lambda F'_1\), where \(\lambda\in \mathbb{C},\) yields that \(S_H(K_1)\) is analytic. By the definition of \(S_H(K_1),\) we obtain that if \(\mathcal{J}_1=\mathcal{J}_2,\) then \(S_H(K_1)=S_H(K_2).\) This means that the 	dilatation \(p_2\) of \(K_2\) is also a constant. So \(G'_2=\frac{\beta}{\alpha}F'_2, \) where \(\alpha,\beta \in \mathbb{C} \) are constants.  The harmonic mappings \(K_1,K_2\) have equal Jacobians. Then we get
\begin{align*}
\left|\frac{F_2'}{F_1'}\right|^2=\frac{1-|\lambda|^2}{1-\left|\frac{\beta}{\alpha}\right|^2}.
\end{align*}
Let \(|\alpha|^2-|\beta|^2=1-|\lambda|^2,\)  we obtain  \(F'_2=\alpha F'_1\) and \(G'_2=\beta F'_1.\)\\
\indent (iii) Similarly, if \(G'_1=\lambda F'_1\), where \(\lambda\in \mathbb{C},\) and \(S_H(K_1)=S_H(K_2),\) then \(S(F_1)=S(F_2).\)
A straightforward calculation shows that  \(F_2=\mathcal{T}\circ F_1 \) if and only if \(S(F_1)=S(F_2), \) where
	\begin{align*}\mathcal{T}(z)=\frac{mz+n}{sz+d},\quad z\in \mathbb{C}, \quad md-ns\neq0.\end{align*}
Therefore we have \(F'_2=(\mathcal{T}\circ F_1)'\) and \(G'_2=c(\mathcal{T}\circ F_1)',\) where \(c\in \mathbb{C} \) with \(|c|<1\) is a constant.
\end{proof}
We now treat the general case of the harmonic mapping 
\(F+\overline{G}\)
where 
\(F'\) and
\(G'\) 
are linearly independent.
 Moreover, we also obtain the properties of the sense-preserving harmonic functions with equal pre-Schwarzian derivatives.
\begin{theorem}\label{lem2}
	Let \(K_1=F_1+\overline{G_1}\) and \(K_2=F_2+\overline{G_2}\)  be two sense-preserving harmonic mappings in a simply connected domain \(\Omega_0\subset \mathbb{C}\) with non-constant  dilatations \( p_1=G'_1/F'_1 \) and \( p_2=G'_2/F'_2 .\)  Set \(\mathcal{J}_1,\mathcal{J}_2\) be the  Jacobians of the harmonic mappings \(K_1\) and \( K_2\), respectively. There are the following properties:
	\begin{enumerate}
		\item[\textup{(i)}] \(P_H(K_1)=P_H(K_2)\) if and only if \(\mathcal{J}_1=c\mathcal{J}_2\) for some constant \(c>0\);
		
		\item[\textup{(ii)}] 
		If   \(P_H(K_1)=P_H(K_2),\) and $F'_1$ and $G'_1$ are linearly independent, then 
		there are two constants \(\alpha,\beta \in\mathbb{C}\) and a real number \(\gamma\) such that 
		\begin{align}\label{e1}
			\begin{array}{cc}
				\begin{pmatrix}
					F'_2 \vspace{10pt} \\
					G_2
				\end{pmatrix}
				=
				\begin{pmatrix}
					\alpha & \beta \mathrm{e}^{i\gamma}\vspace{10pt}\\
					c\overline{\beta} & c^{-1}\overline{\alpha}\mathrm{e}^{i\gamma}
				\end{pmatrix}
				\begin{pmatrix}
					F'_1 \vspace{10pt} \\
					G'_1
				\end{pmatrix}
			\end{array}
		\end{align}
		where \(|\alpha|^2=c(1+c|\beta|^2)\) with constant \(c>0.\)
	\end{enumerate}
\end{theorem}
	\begin{proof} 
		(i) If \(\mathcal{J}_1=c\mathcal{J}_2, \) then we have
	\begin{align*}
		P_H(K_1)=\frac{\partial}{\partial z}\log \mathcal{J}_1
		=\frac{\partial}{\partial z}\log (c\mathcal{J}_2)
		=\frac{\partial}{\partial z}\log \mathcal{J}_2
		=P_H(K_2).
	\end{align*}
	Moreover, if \(P_H(K_1)=P_H(K_2),\) then 
	\begin{align}\label{q1}
		\frac{F''_1}{F_1'}-\frac{p_1'\overline{p_1}}{1-|p_1|^2}=\frac{F''_2}{F_2'}-\frac{p_2'\overline{p_2}}{1-|p_2|^2},
	\end{align}
	which implies that
		\begin{align*}
		\int\frac{F''_1}{F_1'}\ud z-\int\frac{p_1'\overline{p_1}}{1-|p_1|^2}\ud z=\int\frac{F''_2}{F_2'}\ud z -\int\frac{p_2'\overline{p_2}}{1-|p_2|^2}\ud z +\mathcal{C}(\overline{z}).
	\end{align*}
	It is not difficult to find that
	\begin{align}\label{q2}
		\log(|F_1'|^2)+\log(1-|p_1|^2)=\log(|F_2'|^2)+\log(1-|p_2|^2)+\mathcal{C}(\overline{z}).
	\end{align} 
	Next, we will prove \(\mathcal{C}(\overline{z})\) is a constant.  For \eqref{q2}, we take derivatives with respect to \(\bar{z}\) to obtain 
\begin{align*}
		\overline{\left(\frac{F''_1}{F_1'}\right)}-\frac{p_1\overline{p'_1}}{1-|p_1|^2}=\overline{\left(\frac{F''_2}{F_2'}\right)}-\frac{p_2\overline{p'_2}}{1-|p_2|^2}+\partial_{\bar{z}}\mathcal{C}(\overline{z}).
\end{align*}
By \eqref{q1}, we deduce that \(\partial_{\bar{z}}\mathcal{C}(\overline{z})=0.\) Now, we let \(\mathcal{C}=\log c\) with constant \(c>0.\) From \eqref{q2},  we get \(\mathcal{J}_1=c\mathcal{J}_2. \)

(ii) As we mentioned before, \(P_H(K_1)=P_H(K_2)\) implies \(\mathcal{J}_1=c\mathcal{J}_2\) for some constant \(c>0.\) Then we have
\begin{align}\label{d1}
	|F'_1|^2-|G'_1|^2=c(|F'_2|^2-|G'_2|^2).
\end{align}
By assumption, then we get \(|F'_1|>0.\)
After dividing \eqref{d1} by \(|F'_1|^2,\) we see that
\begin{align}\label{d2}
	1-|w_1|^2=c(|w_2|^2-|w_3|^2),
\end{align}
where \(w_1=G'_1/F'_1,\) \(w_2=F'_2/F'_1\) and \(w_3=G'_2/F'_1\) are analytic. Take the Laplacian of both sides of \eqref{d2} to get
\begin{align*}
		|w'_1|^2=c(|w'_3|^2-|w'_2|^2).
\end{align*}
Next, we will prove \(w'_2\neq0.\) Assume that \(w_2\) equals a constant \(m_0/\sqrt{c},\) then we rewrite \eqref{d2} as
\begin{align}\label{d4}
		1-|m_0|^2+c|w_3|^2=|w_1|^2.
\end{align}
 If \(|m_0|=1,\) then \(F'_1=\sqrt{c}\mathrm{ e}^{-il_1}F'_2\) and \(G'_1=\sqrt{c}\mathrm{ e}^{il_2}G'_2,\)  which satisfies \eqref{e1}.
 If \(|m_0|\neq1,\)
we take logarithms in \(\eqref{d4}\) to get
\begin{align}\label{d5}
	\log(|w_1|^2)=\log(1-|m_0|^2+c|w_3|^2).
\end{align}
Applying the operator \(\Delta=4\partial_z\partial_{\bar{z}}\) to  \eqref{d5} yields that
\begin{align*}
	\Delta\left(\log(1-|m_0|^2+c|w_3|^2)\right)&=\frac{4c(1-|m_0|^2)|w'_3|^2}{(1-|m_0|^2+c|w_3|^2)^2}\\
	&=0,
\end{align*}
since the function the left-hand side of \eqref{d5} is harmonic.
Thus,  \(w_3'=0, \) and  \eqref{d4} implies that \(w'_1=0, \) which contradicts with our assumptions. Because we assume that $F'_1$ and $G'_1$ are linearly independent, then \(w'_1\neq0. \) Therefore, we get \(w'_2\neq0.\) 

Denote \(n_1=\frac{w'_1}{\sqrt{c}w'_2}\) and \(n_2=\frac{w'_3}{w'_2},\) then
\begin{align}\label{d6}
	|n_2|^2=1+|n_1|^2.
\end{align}
As before, we take logarithms in \eqref{d6} to get
\[\log(|n_2|^2)=\log(1+|n_1|^2).\]
Taking the  Laplacian, we have
\begin{align*}
	\Delta\left(\log(1+|n_1|^2)\right)&=\frac{|n_1'|^2}{(1+|n_1|^2)^2}\\
	&=0.
\end{align*}
So \(n_1\) is a constant, and by \eqref{d6} \(n_2\) is also a constant.  We set 
\[n_1=n_0\mathrm{e}^{i\theta_1}, \quad n_2=\sqrt{1+n_0^2} \mathrm{e}^{i\theta_2},\] 
for some real constants \(n_0, \theta_1, \theta_2.\) Then
\begin{align}\label{dd1}
	w_1=\sqrt{c}n_0\mathrm{e}^{i\theta_1}w_2 + s_1,
\end{align}
and
\begin{align}\label{dd2}
	w_3=\sqrt{1+n_0^2}\mathrm{e}^{i\theta_2}w_2 + s_2.
\end{align}
From \eqref{d2}, we obtain
\begin{align}\label{d7}
	2\text{Re}\left\{w_2\left(\sqrt{c}n_0\mathrm{e}^{i\theta_1}\overline{s_1}-c\sqrt{1+n_0^2}\mathrm{e}^{i\theta_2}\overline{s_2}\right)\right \}=1+c|s_2|^2-|s_1|^2.
\end{align}
 As we mentioned before, this is not possible for non-constant \(w_2.\) Note that the right-hand side of \eqref{d7} equals to a constant, thus
 \begin{align}\label{d8}
 		s_1=\frac{\sqrt{c(1+n_0^2)}}{n_0}\mathrm{e}^{i(\theta_1-\theta_2)}s_2,
 \end{align}
 and
 \begin{align}\label{d9}
 			|s_1|^2=1+c|s_2|^2,
 \end{align}
 Moreover, from \eqref{d8} and \eqref{d9}, we have
 \begin{align*}
 |s_2|=\sqrt{c}n_0.
 \end{align*}
  Setting \(s_2=\sqrt{c}n_0\mathrm{e}^{i\theta_3},\)
 by \eqref{dd1}, then
 \begin{align*}
 	\frac{G_1'}{F'_1}=\sqrt{c}n_0\mathrm{e}^{i\theta_1}\frac{F'_2}{F'_1}+s_1.
 \end{align*}
 Hence
 \begin{align*}
 	F'_2=\frac{1}{\sqrt{c}n_0\mathrm{e}^{i\theta_1}}G'_1-\frac{\sqrt{c(1+n_0^2)}}{n_0}\frac{\mathrm{e}^{i\theta_3}}{\mathrm{e}^{i\theta_2}}F'_1.
 \end{align*}
In light of \eqref{dd1} and \eqref{dd2}, then
 \begin{align*}
 G'_2&=\sqrt{1+n_0^2}\mathrm{e}^{i\theta_2}F'_2+s_2F'_1\\
 &=\sqrt{1+n_0^2}\mathrm{e}^{i\theta_2}\left(\frac{1}{\sqrt{c}n_0\mathrm{e}^{i\theta_1}}G'_1-\frac{\sqrt{c(1+n_0^2)}}{n_0}\frac{\mathrm{e}^{i\theta_3}}{\mathrm{e}^{i\theta_2}}F'_1\right)+s_2F'_1\\
 &=\frac{\sqrt{1+n_0^2}}{\sqrt{c}n_0}\frac{\mathrm{e}^{i\theta_2}}{\mathrm{e}^{i\theta_1}}G'_1-\frac{\sqrt{c}}{n_0}\mathrm{e}^{i\theta_3}F'_1.
 \end{align*}
 Finally, setting \(\gamma=\theta_3-\theta_1+\pi\) and \(\theta_2=\theta_1+2\pi\),   denoting 
 \[\alpha=-\frac{\sqrt{c(1+n_0^2)}}{n_0}\mathrm{e}^{i(\theta_3-\theta_2)},  \quad\beta=-\frac{\mathrm{e}^{-i\theta_3}}{\sqrt{c}n_0},\]
 then
 \begin{align*}
 	\begin{cases}
 		F_2'=\alpha F'_1+\mathrm{e}^{i\gamma}\beta G'_1\\[5pt]
 		G_2'=c\overline{\beta}F'_1+\frac{1}{c}\mathrm{e}^{i\gamma}\overline{\alpha}G'_1,
 	\end{cases}
 \end{align*}
 which we complete the proof of the theorem.
	\end{proof}
\begin{remark}
	In light of \eqref{q1}, we deduce that
	\begin{align*}
		\partial_tP_H(K)=\frac{\partial^2}{\partial tz}\log \mathcal{J}=\frac{\partial}{\partial z} \frac{\mathcal{J}_t}{\mathcal{J}}=0,
	\end{align*}
	and 
	\begin{align*}
			\partial_tS_H(K)=\frac{\partial^2}{\partial tz}\left(P_H(K)\right)-P_H(K)\partial_tP_H(K)=0.
	\end{align*}
When the mass conservation equation \eqref{1.2} is expressed in terms of Lagrangian variables, the Schwarzian  and pre-Schwarzian derivatives for the harmonic mapping \eqref{1.15} are time-independent. This allows us to apply Theorems \ref{lem3.1} and \ref{lem2} to simplify Eq.\eqref{4.1}.
\end{remark}
\begin{remark}
Theorems \ref{thm4}, \ref{lem3.1} and \ref{lem2} can also be applied to study the incompressible Euler equations, as the incompressibility implies that  Schwarzian and pre-Schwarzian derivatives of harmonic mapping \eqref{1.15} are time-independent, and the obtained results are consistent with Ref.\cite{r19}.
\end{remark}

    \section{A Simple Class of Solutions }\label{sect4}
    In this section, we begin by treating the simpler case when  \(	\partial^2_t\omega=0.\) 	 
    Denote
    \begin{align*}
     \delta=\frac{1}{2}\left( x_{att}x_b+y_{att}y_b-x_{btt}x_a-y_{btt}y_a\right),
    \end{align*}
     then we find that
     \begin{align*}
    \mathcal{J}\partial_t\omega=2\delta,
    \end{align*}
     and further
    \begin{align}\label{x}
    	\frac{\partial \delta}{\partial t}=0,
    \end{align}
    which means \(\delta=\delta(z,\bar{z}).\)
    From \eqref{c1}-\eqref{c2}, we find that
    \begin{align*}
    	\left(F'_t(t,z)\overline{F'(t,z)}-\overline{G'_t(t,z)}G'(t,z)\right)_t=i\delta(z,\bar{z}).
    \end{align*} 
    Integrating the above equation from 0 to $t$ yields that
    \begin{align}\label{1.24}
    	F'_t(t,z)\overline{F'(t,z)}-\overline{G'_t(t,z)}G'(t,z)=i\delta(z,\bar{z})t+\rho(z,\bar{z}).
    \end{align}
    Recall that
    \begin{align}
    	\left(|F'|^2-|G'|^2\right)_t
    	&=\left(F'_t\overline{F'}-\overline{G'_t}G'\right)_t+
    	\overline{\left(F'_t\overline{F'}-\overline{G'_t}G'\right)_t}\nonumber\\
    	&=i\delta t+\rho-i\bar{\delta}t+\bar{\rho}\nonumber\\
    	&=i(\delta-\bar{\delta} )t+(\rho+\bar{\rho})\nonumber\\
    	&=0,
    \end{align}
    which implies
    \begin{equation*}
    	\begin{cases}
    		\delta=\bar{\delta},\\
    		\rho=-\bar{\rho}.
    	\end{cases}
    \end{equation*}
    This means that $\delta(z,\bar{z})$ is a real function, and $\rho(z,\bar{z})$ is a purely imaginary. There is a real function $\sigma(z,\bar{z})$ satisfying $\rho(z,\bar{z})=i\sigma(z,\bar{z})$.  Therefore,  $\eqref{1.24}$ becomes
    \begin{align}\label{1.26}
    	f_t(t,z)\overline{f(t,z)}-\overline{g_t(t,z)}g(t,z)=
    	i\left(\delta(z,\bar{z})t+\sigma(z,\bar{z})\right).
    \end{align}
    Using \eqref{R}, then  \(\eqref{1.26}\) becomes
    \begin{align}\label{2.13}
    	\mathcal {L} f- \mathcal {L} \overline{g}=i(\delta(z,\bar{z})t+\sigma(z,\bar{z})),
    \end{align}
    where \(\delta(z,\bar{z})\) and \(\sigma(z,\bar{z})\) are real functions in \(\Omega_0\times \Omega^*\) with
    \begin{align*}
    \Omega^*=\{z\in \mathbb{C}:\overline{z}\in \Omega_0\}.\
    \end{align*}
Let \(\overline{z}=\xi\), then we obtain the equation 
    		\begin{align}\label{2.14}
    		f_t(t,\overline{\xi})\overline{f(t,\overline{\xi})}-\overline{g_t(t,\overline{\xi})}g(t,\overline{\xi})=
    		i\left(\delta(\overline{\xi},\xi)t+\sigma(\overline{\xi},\xi)\right),\quad \xi\in\Omega^*,
    	\end{align}
  which is obviously equivalent to \(\eqref{1.26}.\) 
  Notice that for \(\xi_1\neq \xi_2\) in \(\Omega^*\), from \(\eqref{2.14}\) we obtain a linear system in the unknown functions \(f (t,\overline{\xi})\) and \(ig_t(t,\overline{\xi})\). This allows us to find \(\xi_1\neq \xi_2\) in \(\Omega^*\) such that the system \(\eqref{2.14}\) is nonzero at some time \(t_0>0\),  and on an open interval \(I\subset[0,\infty)\), so that the two vectors \( {Q_1}\) and \({Q_2}\) are linearly independent, where 
  \[
  \begin{array}{cc}
   Q_j=
  
  \begin{pmatrix}
  	f_t(t,\overline{\xi_j})  \vspace{10pt}\\
  	-ig(t,\overline{\xi_j})
  \end{pmatrix},
\end{array}
\quad j=1,2.
  \]
  The system \(\eqref{2.14}\) can be re-written 
   as
  \[\begin{pmatrix}
  	f_t(t,\overline{\xi_1})&-ig(t,\overline{\xi_1})  \vspace{10pt}\\
  	f_t(t,\overline{\xi_2})&-ig(t,\overline{\xi_2})
  \end{pmatrix}\begin{pmatrix}
  	\overline{f(t,\overline{\xi})}  \vspace{10pt}\\
  	-i\overline{g_t(t,\overline{\xi})}
  \end{pmatrix}=i\begin{pmatrix}
 \delta(\overline{\xi},\xi_1
  )\vspace{10pt}\\
  \delta(\overline{\xi},\xi_2)
  \end{pmatrix}t+\begin{pmatrix}
 \sigma(\overline{\xi},\xi_1)\vspace{10pt}\\
 \sigma(\overline{\xi},\xi_2)
  \end{pmatrix},\quad t\in I,\quad \xi\in\Omega^*.
  \]
  Since  the two vectors \( {Q_1}\) and \({Q_2}\) are linearly independent, then we
  denote the matrix 
  \[\begin{pmatrix}
  	T_1(t)&T_2(t) \vspace{10pt} \\
  	T_3(t)&T_4(t)
  \end{pmatrix}\]
  being the inverse of the matrix 
  \[\begin{pmatrix}
  	f_t(t,\overline{\xi_1})&-ig(t,\overline{\xi_1}) \vspace{10pt} \\
  	f_t(t,\overline{\xi_2})&-ig(t,\overline{\xi_2})
  \end{pmatrix}.\]
  Let 
  \begin{align*}\delta_j(\overline{\xi})=\delta(\overline{\xi},\xi_j)\quad \text{and } \quad \sigma_j(\overline{\xi})=\sigma(\overline{\xi},\xi_j)\quad j=1, 2,
  \end{align*}
  then we  can transform the above equation into the linear system
  \begin{align*}
  	\begin{pmatrix}
  		\overline{f(t,\overline{\xi})} \vspace{10pt} \\
  		-i\overline{g_t(t,\overline{\xi})}
  	\end{pmatrix}=\begin{pmatrix}
  	T_1(t)&T_2(t) \vspace{10pt} \\
  	T_3(t)&T_4(t)
  	\end{pmatrix}\begin{pmatrix}
  	i\delta_1(\overline{\xi})t+\sigma_1(\overline{\xi}) \vspace{10pt} \\
  		i\delta_2(\overline{\xi})t+\sigma_2(\overline{\xi})
  	\end{pmatrix},
  \end{align*}
  where 
  \begin{align*}
  T_1(t)T_4(t)-T_2(t)T_3(t)\neq 0,\quad \text{for } t\in I.
  \end{align*}
  Similarly, for the system \(\eqref{1.26}\) in the unknown functions \(f_t(t,z)\) and \(-ig(t,z)\), we choose the two appropriate vectors 
   \[
  \begin{array}{cc}
  	 H_j=
  	\begin{pmatrix}
  		f(t,z_j)  \vspace{10pt}\\
  		ig_t(t,z_j)
  	\end{pmatrix},
  \end{array}
  \quad j=1,2.
  \]
  There exist \(z_1\neq z_2\) in \(\Omega_0\) and an open interval \(I_1\subset[0,\infty)\) such that for \(t\in I_1\) two vectors \({H_1}\) and \({H_2}\) are linearly independent.
  Therefore we can recast the equation \eqref{2.14} as \[\begin{pmatrix}
  	\overline{f(t,z_1)}&-i\overline{g_t(t,z_1)}\vspace{10pt}  \\
  	\overline{f(t,z_2)}&-i\overline{g_t(t,z_2)}
  \end{pmatrix}\begin{pmatrix}
  	f_t(t,z) \vspace{10pt} \\
  	-ig(t,z)
  \end{pmatrix}=i\begin{pmatrix}
  	\delta(z,\bar{z}_1)\vspace{10pt}\\
  	\delta(z,\bar{z}_2)
  \end{pmatrix}t+\begin{pmatrix}
  	\sigma(z,\bar{z}_1)\vspace{10pt}\\
  	\sigma(z,\bar{z}_2)
  \end{pmatrix},\quad t\in I_1,\quad z\in\Omega_0.
  \]
  According to the above considerations,
  it is not difficult to find that the function  \(f(t,z)\) (or \(g(t,z)\)) has the form
  \begin{align*}
  h_1(z)Y_1(t)+h_2(z)Y_2(t),\quad Y_1,Y_2\in C^1( I_1),
  \end{align*}
  for two appropriate linearly independent functions \(h_1,h_2\in C^1(\Omega_0).\)
  
  Moreover, another possibility is that two vectors 
   \[
  \begin{array}{cc}
  	 H=
  	\begin{pmatrix}
  		f(t,z) \vspace{10pt} \\
  		ig_t(t,z)
  	\end{pmatrix},
\quad
 Q=
  	\begin{pmatrix}
  		f_t(t,\overline{\xi})  \vspace{10pt}\\
  		-ig(t,\overline{\xi})
  	\end{pmatrix},
  \end{array}
  \]
   are linearly dependent for any instant \(t\in [0,\infty)\) and \(z,\overline{\xi} \in\Omega_0\), that is 
  \begin{align}
  	f_t(t,\overline{\xi})=A(t)g(t,\overline{\xi}), \quad t\geq0,\quad\overline{\xi}\in \Omega_0, \label {k1}\vspace{10pt}\\
  		g_t(t,z)=B(t)f(t,z),\quad t\geq0,\quad z\in \Omega_0,\label{k2}
  \end{align}
  for two  functions \(A,B: [0,\infty)\mapsto \mathbb C\setminus\{0\}\) of class \(C^1.\) The equation \(\eqref{1.26}\) becomes
  \begin{align}\label{k3}
  	(A(t)-\overline{B(t)})\overline{f(t,z)}g(t,z)=  		i\left(\delta(z,\bar{z})t+\sigma(z,\bar{z})\right),\quad t\geq0,\quad z\in\Omega_0.
  \end{align}
  If \(f=0\) then from \(\eqref{k1}\) we deduce that \(g_t\) is time-independent so that \(G(t,z)=G_0(z)\), and if \(g=0\) then \(f_t\) is also time-independent  so that \(F(t,z)=F_0(z).\) However by \( \eqref{1.15}\) we can see this result  is  trivial. Therefore we should seek solutions \( f\neq 0\) and \( g\neq 0\). 
  \begin{theorem}\label{t1}
  If  \( f\neq 0\) and \( g\neq 0\) satisfy  \eqref{k3}, then
  either \(A(t)=\overline{B(t)}\) or \(f(t,z)=\varrho(t)g(t,z)\)  and \(A(t)\neq \overline{B(t)}\) for all \(t\geq0,\)
  where \(\varrho : [0, \infty)\mapsto \mathbb{C}\setminus \{0\}\) is a \(C^1\) function.
  \end{theorem}
  \begin{proof}
  		Differentiating \(\eqref{k3}\) with respect to \(t\) yields that
  	\begin{align}\label{k4}
  		\left(A'(t)-\overline{B'(t)}\right)\overline{f(t,z)}g(t,z)+\left(A(t)-\overline{B(t)}\right)\left(\overline{f_t(t,z)}g(t,z)
  		+\overline{f(t,z)}g_t(t,z)\right)
  		=  		i\delta(z,\bar{z}).
  	\end{align}
  	Assume that there exists \(t_0\in(0,\infty)\) so that \(A(t_0)=\overline{B(t_0)}\) and \(A'(t_0)=\overline{B'(t_0)}\). Otherwise, we can deduce that the open set $\{\tau>0:A(\tau)\neq \overline{B(\tau)}, A'(\tau)\neq \overline{B'(\tau)}\}$ is nonempty, and has an open subset \(I_0\). However  evaluating \(\eqref{k4}\) at \(t=t_0\) yields \( \delta(z,\bar{z})=0\) holds for all \(t\geq0.\) Then for every \(t\in I_0\) we can deduce that \(\overline{f(t,z)} g(t,z)=0\) and \(\left(\overline{f(t,z)} g(t,z)\right)_t=0.\) Therefore we obtain the contradiction  \(f=0\) or \(g=0.\) If \(A(t)\neq \overline{B(t)}\) for all \(t>0\), dividing \(\eqref{k4}\) by \(|f(t,z)|^2\) yields that
  	\begin{align}\label{k5}
  		\left(A'(t)-\overline{B'(t)}\right)\frac{g(t,z)}{f(t,z)}+\left(A(t)-\overline{B(t)}\right)\frac{\overline{f_t(t,z)}g(t,z)
  		+\overline{f(t,z)}g_t(t,z)}{|f(t,z)|^2}
  		=  		i\frac{\delta(z,\bar{z})}{|f(t,z)|^2}.
  	\end{align}
In light of \(\eqref{k1}\) and  \(\eqref{k1}\), then the equation \(\eqref{k5}\) becomes
  	\begin{align}\label{k6}
	\left(A'(t)-\overline{B'(t)}\right)\frac{g(t,z)}{f(t,z)}+\left(A(t)-\overline{B(t)}\right)\left(\overline{A(t)}\left| \frac{g(t,z)}{f(t,z)}\right|^2+B(t)\right)
	=  		i\frac{\delta(z,\bar{z})}{|f(t,z)|^2}.
\end{align}
 \(f(t,z)\neq 0\) allows us to choose \( z_0 \in\Omega_0\) such that for \(z \in \mathcal{B}(z_0,\delta_1),\) where \(\mathcal{B}(z_0,\delta_1)=\{z\in\Omega_0: |z-z_0|<\delta_1\},\) such that \(f(t,z)\neq 0.\) By the open mapping theorem, the analytic map \(z\mapsto\frac{g(t,z)}{f(t,z)}\) is not  open in \(\mathcal B(z_0;\delta_1)\)  unless it is  a constant.  Hence for \(z\in\mathcal B(z_0;\delta_1) \) we have \(f(t,z)=\varrho(t)g(t,z)\) with \(\varrho : [0, \infty)\mapsto \mathbb{C}\setminus \{0\}\) of class \(C^1.\) Finally, we apply the identity theorem to obtain that \(f(t,z)=\varrho(t)g(t,z)\) in \( [0,\infty) \times \Omega_0.\)
  \end{proof}
  
      \begin{theorem}\label{t2}
      	Assume that \(A(0)\neq \overline{B(0)}.\)
  	Both \(F_0(z)\) and \(G_0(z)\) are univalent in \(\Omega_0\) if the harmonic mapping \(z\mapsto F_0(z)+\overline{G_0(z)}\) is univalent in  \(\Omega_0.\)  
  \end{theorem}
  \begin{proof}
  	According to the previous considerations, the map \(K(z)= F_0(z)+\overline{G_0(z)}\) is univalent in \(\Omega_0,\)  such that there are two different complex numbers \(z_1,z_2 \in \Omega_0\) satisfying \(K(z_1)\neq K(z_2).\) Since \(A(0)\neq \overline{B(0)},\) according to Theorem \(\ref{t1}, \) we obtain that \(f_0(z)=\varrho(0)g_0(z). \) Then by means of the definitions of \(f_0\) and \(g_0,\) we have \(F_0(z)=\varrho(0)G_0(z). \) If the map \(z\mapsto F_0(z)\) is not univalent in \( \Omega_0\),  the map \(z\mapsto G_0(z)\) is also not univalent in \( \Omega_0,\) as \(\varrho(0) \) is different from 0. In other words, there exists \(z_1\neq z_2\) such that \(F_0(z_1)=F_0(z_2)\) and \(G_0(z_1)=G_0(z_2).\)  Then we obtain
  	 \[F_0(z_1)-F_0(z_2)+\overline{G_0(z_1)-G_0(z_2)}=0,\]
  	which contradicts with the univalence of \(K(z).\)
  \end{proof}

    \subsection{The linearly dependent case }
    In this subsection, from Theorem \ref{t1}, we can see that \(A(t)=\overline{B(t)}\) or \(f(t,z)=\varrho(t)g(t,z)\) when two vectors \(H\) and \(Q\) are linearly dependent. Next, we will analyze the solutions to the equation \eqref{k3} separately for these two cases.  
    
    When \(A(t)=\overline{B(t)},\)  then we get
    \begin{align}\label{k7}
    	\begin{cases}
    		f_t(t,z)=\overline{B(t)}g(t,z), \quad t\geq0,\quad z \in \Omega_0, \\
    	g_t(t,z)=B(t)f(t,z),\quad t\geq0,\quad z\in \Omega_0,
    	\end{cases}
    \end{align}
     where \(B: [0,\infty)\mapsto \mathbb C\setminus\{0\}\) is a \(C^1\) function.  
    Even if we can transform this system \(\eqref{k7}\) into a  second-order linear differential equation
    \begin{align*}
    	f_{tt}(t,z)=\frac{\overline{B'(t)}}{\overline{B(t)}}f_t(t,z)+|B(t)|^2f(t,z)
    \end{align*}
    	with initial data \(f(0,z)=f_0(z)\) at every fixed \(z \in \Omega_0,\)  
    	 for any  \(C^1\) complex function \(B(t)\)  it is difficult to find the explicit form of \(f(t,z)\). For example, if \(B(t)\) is a real function and strictly positive,
    	 then we can 
    	 denote \(L=f/\sqrt{B},\) such that the above equation can be transformed into
    	 \begin{align*}
    	 	L_{tt}=\frac{f}{\sqrt{B}}\left(B^2-\frac{B_{tt}}{2B}+\frac{3B_t^2}{4B^2}\right)=ML.
    	 \end{align*}
    	 Alternatively, we can also write it as
    	\begin{align}\label{k9}
    	\mathlarger{\frac{\ud}{\ud t}} \begin{pmatrix} L(t,z) \vspace{10pt}\\ W(t,z) \end{pmatrix}
    	= 
    	\begin{pmatrix} 0 & 1 \vspace{10pt}\\ M(t) & 0 \end{pmatrix}
    	\begin{pmatrix} L(t,z) \vspace{10pt}\\ W(t,z) \end{pmatrix}
    	\end{align}
    	 with the initial data 
    	 \[
    	 \begin{pmatrix} L(0,z) \vspace{10pt}\\
    	 	 W(0,z) \end{pmatrix}
    	 = 
    	 \begin{pmatrix} 
    	 	\frac{f_0(z)}{\sqrt{B(0)}} \vspace{10pt}\\  0 \end{pmatrix}
    	 .\]
   However if we need to find all general explicit solutions of the system \eqref{k9}, in view of the form of \(M(t)\), it is difficult to determine whether this coefficient matrix is singular. Therefore we can choose the particular function \(B(t)\) satisfying
   \[\int_{0}^{t}B(s)\ud s=r(t)\mathrm{e}^{i\Theta (t)}\] with \(\Theta:[0,\infty)\mapsto \mathbb{R}\) and \(r:[0,\infty)\mapsto (0,\infty)\) of class \(C^1.\) Then by means of the Lemma \ref{lem3},
   for special  complex function \(B(t)\), we can find the solution \(f(t,z)\) explicitly. The linear system \(\eqref{k7}\) can be expressed as
   	\begin{align}\label{k8}
   \mathlarger{\frac{\ud}{\ud t}} \begin{pmatrix} f(t,z)\vspace{10pt} \\ g(t,z) \end{pmatrix}
   = 
   \begin{pmatrix} 0 & \overline{B(t)}\vspace{8pt} \\ B(t) & 0 \end{pmatrix}
   \begin{pmatrix} f(t,z)\vspace{10pt} \\ g(t,z) \end{pmatrix}
   \end{align}
   with the initial data 
   \[
   \begin{pmatrix} f(0,z)\vspace{10pt} \\
   	g(0,z) \end{pmatrix}
   = 
   \begin{pmatrix} 
   	f_0(z)\vspace{10pt} \\  g_0(z) \end{pmatrix}
   .\]
   Let \(X(t)\) and \( D(t)\) be the fundamental 
   solution matrix and coefficient matrix of the linear different system \eqref{k8}, respectively. Then we have
   \begin{align}\label{v1}
   	\begin{cases}
   		X'(t)=D(t)X(t),\\
   		X(0)=I_1.
   	\end{cases}
   \end{align}
   where \[ I_1=
   \begin{pmatrix} 
   	1&0 \\  0&1 \end{pmatrix}
   .\] Then the solutions can be expressed as
    \[
   \begin{pmatrix} f(t,z)\vspace{10pt} \\
   	g(t,z) \end{pmatrix}
   = X(t)
   \begin{pmatrix} 
   	f_0(z) \vspace{10pt}\\  g_0(z) \end{pmatrix}
   .\]
   
  In order to find all solutions of  system \eqref{v1}, we will make use of the following two lemmas.
    \begin{lemma}\label{lem3}(\cite{r5})
    	If the coefficient matrix \(D(t)\) is analytic and non-singular, then for the system \eqref{v1} the fundamental 
    	solution matrix \(X(t)\) has the representation: \(X(t)=\exp(\int_{0}^{t}D(s)\ud s)\)  if and only if \(D(t)\) commutes with \(\int_{0}^{t} D(s)\ud s.\)
    \end{lemma}
   
    \begin{lemma}\label{lem4}
    	Let \(\mathcal{A}\) and \(\mathcal{C}\) be two  \(n\times n\) matrices.   \(\mathcal{A}\) commutes with \(\mathcal{C}\)  if and only if \[\mathrm{e}^{\mathcal{A}}\cdot\mathrm{e}^{\mathcal{C}}=\mathrm{e}^{\mathcal{C}}\cdot\mathrm{e}^{\mathcal{A}}=\mathrm{e}^{\mathcal{A}+\mathcal{C}}.\]
    \end{lemma}
    \begin{proof}
    	If \(\mathcal{A}\mathcal{C}=\mathcal{C}\mathcal{A},\) we say that \(\mathcal{A}\) and \(\mathcal{C}\) commute \cite{r6}.
    	For all \(t\geq0,\) we have 
    	\begin{align*}
    	\mathrm{e}^{\mathcal{A}t}\cdot\mathrm{e}^{\mathcal{C}t}&=\left(I_2+\mathcal{A}t+\frac{\mathcal{A}^2t^2}{2}+\cdots\right)\left(I_2+\mathcal{C}t+\frac{\mathcal{C}^2t^2}{2}+\cdots\right)\\
    	&=I_2+(\mathcal{A}+\mathcal{C})t+\frac{\mathcal{A}^2t^2}{2}+\mathcal{A}\mathcal{C}t+\frac{\mathcal{C}^2t^2}{2}+\cdots,
    	\end{align*}
    	where \(I_2\) is  \(n\times n\) identity matrix.
    	Moreover, since 
    \begin{align*}
    	\mathrm{e}^{(\mathcal{A}+\mathcal{C})t}=I_2+(\mathcal{A}+\mathcal{C})t+\frac{(\mathcal{A}+\mathcal{C})^2t^2}{2}+\cdots,
    	\end{align*}
    then we find that
    \begin{align*}
    	\mathrm{e}^{\mathcal{A}t}\cdot\mathrm{e}^{\mathcal{C}t}-\mathrm{e}^{(\mathcal{A}+\mathcal{C})t}&=\left(\mathcal{A}\mathcal{C}-\mathcal{C}\mathcal{A}\right)\frac{t^2}{2}+\cdots.
    \end{align*}
    Consequently, \(\mathrm{e}^{\mathcal{A}t}\cdot\mathrm{e}^{\mathcal{C}t}=\mathrm{e}^{(\mathcal{A}+\mathcal{C})t}\) for all \(t\geq0\) if and only if  \(\mathbf{\mathcal{A}}\) and \(\mathcal{C}\) commute.
    It is not difficult to check that 
    \begin{align*}
    	\mathrm{e}^{\mathcal{C}t}\cdot\mathrm{e}^{\mathcal{A}t}&=\left(I_2+\mathcal{C}t+\frac{\mathcal{C}^2t^2}{2}+\cdots\right)\left(I_2+\mathcal{A}t+\frac{\mathcal{A}^2t^2}{2}+\cdots\right)\\
    	&=I_2+(\mathcal{A}+\mathcal{C})t+\frac{\mathcal{A}^2t^2}{2}+\mathcal{C}\mathcal{A}t+\frac{\mathcal{C}^2t^2}{2}+\cdots\\
    	&=\mathrm{e}^{\mathcal{A}t}\cdot\mathrm{e}^{\mathcal{C}t}
    \end{align*}
    if and only if  \(\mathcal{A}\mathcal{C}=\mathcal{C}\mathcal{A},\) that is \(\mathbf{\mathcal{A}}\) commutes with \(\mathcal{C}.\) 
    \end{proof}
    
    \begin{theorem}\label{thm1}
Let \(\Omega_0\subset \mathbb{C}\) be a simply connected domain. Given the $C^1$-function
\(r : [0,\infty) \mapsto (0,\infty) \)  and arbitrary constant \(k_0\in \mathbb{R}.\) If the linear differential system \eqref{k7} holds, then the particle motion \(\eqref{2.13}\) in a fluid flow is described by
\begin{align}\label{r1}
	\begin{cases}
		f(t,z)=\cosh(r(t))f_0(z)+\mathrm{e}^{-ik_0}\sinh(r(t))g_0(z),~~ z\in \Omega_0,\vspace{10pt}\\
		g(t,z)=\cosh(r(t))g_0(z)+\mathrm{e}^{ik_0}\sinh(r(t))f_0(z),~~z\in \Omega_0.\\
	\end{cases}
\end{align}
    \end{theorem}
    
    \begin{proof}Recall that 
    	\[D(t)=\begin{pmatrix} 0 & \overline{B(t)}\vspace{10pt} \\ B(t) & 0 \end{pmatrix},\]
    	where \begin{align}\label{k12}
    	\int_{0}^{t}B(s)\ud s=r(t)\mathrm{e}^{i\Theta (t)}.
    	 \end{align}
     Notice that from \eqref{k8} we have \(\det(D(t))=-|B(t)|^2<0,\) in view of the continuous dependence of \(B(t)\) on \(t,\)  then the coefficient matrix \(D(t)\) is   non-singular and analytic. It is easy to find that 
     \[D^2(t)=|B(t)|^2\begin{pmatrix} 1 & 0 \vspace{5pt}\\ 0 & 1 \end{pmatrix}=|B(t)|^2I_3.\]
     Furthermore, we find that
     \begin{align*}
     	D^n(t)=
     	\begin{cases}
     		|B(t)|^{2l}I_3,\quad &n=2l, \quad\quad~~ l=0,1,2,\cdots,\\~\\
     		|B(t)|^{2l+1}\frac{D(t)}{|B(t)|},\quad& n=2l+1,\quad l=0,1,2,\cdots.
     	\end{cases}
     \end{align*}
     When \(n=2l,\) then we have
     \begin{align}\label{k10}
     	\sum_{l=0}^{\infty}\frac{D^{2l}(t)}{(2l)!}=\sum_{l=0}^{\infty}\frac{|B(t)|^{2l}(t)}{(2l)!}I_3
     	&=\cosh\left(|B(t)|\right)I_3.
     \end{align}
     When \(n=2l+1,\) then we obtain
          \begin{align}\label{k11}
     	\sum_{l=0}^{\infty}\frac{D^{2l+1}(t)}{(2l+1)!}=\sum_{l=0}^{\infty}\frac{|B(t)|^{2l+1}(t)}{(2l+1)!}\frac{D(t)}{|B(t)|}=\sinh(|B(t)|)\frac{D(t)}{|B(t)|}.
     \end{align}
     Adding \eqref{k10} to \eqref{k11} yields that
     \begin{align*}
     	\exp(D(t))&=\sum_{l=0}^{\infty}\frac{D^{2l}(t)}{(2l)!}+	\sum_{l=0}^{\infty}\frac{D^{2l+1}(t)}{(2l+1)!}\\
     	&=\cosh(|B(t)|)I_3+\sinh(|B(t)|)\frac{D(t)}{|B(t)|}\\
     	&=\cosh(|B(t)|)\begin{pmatrix} 1 & 0 \\ 0 & 1 \end{pmatrix}+\frac{\sinh(|B(t)|)}{|B(t)|}\begin{pmatrix} 0 & \overline{B(t)} \\ B(t) & 0 \end{pmatrix}\\
     	&=\begin{pmatrix}  \cosh(|B(t)|)&\sinh(|B(t)|)\frac{\overline{B(t)}}{|B(t)|}  \vspace{10pt}\\ \sinh(|B(t)|)\frac{{B(t)}}{|B(t)|} & \cosh(|B(t)|) \end{pmatrix}.
     \end{align*}
    Similarly, from \eqref{k12} we get 
\begin{align*}
	\exp\left(\int_{0}^{t}D(s)\ud s\right) &= \exp\begin{pmatrix}
		0 & \int_{0}^{t}\overline{B(s)}\ud s\vspace{10pt} \\
		\int_{0}^{t}{B(s)}\ud s & 0
	\end{pmatrix} \vspace{10pt}\\[5pt]
	&= \exp\begin{pmatrix}
		0 &r(t)\mathrm{e}^{-i\Theta (t)} \vspace{10pt}\\
		r(t)\mathrm{e}^{i\Theta (t)} & 0
	\end{pmatrix}\vspace{10pt} \\[5pt]
	&=\begin{pmatrix}  \cosh(r(t))&\sinh(r(t))\mathrm{e}^{-i\Theta (t)}\vspace{10pt} \\ \sinh(r(t))\mathrm{e}^{i\Theta (t)} & \cosh(r(t)) \end{pmatrix}.
\end{align*}
Let \[\Lambda_1(t)=\sqrt{\left(r(t)+r'(t)\right)^2+r^2(t)\Theta'^2(t)}.\]
     Since \[B(t)+r(t)\mathrm{e}^{i\Theta (t)}=\mathrm{e}^{i\Theta (t)}\left(r(t)+r'(t)+ir(t)\Theta'(t)\right),\] then we deduce that
     \begin{align*}
     	&\exp\left(D(t)+\int_{0}^{t}D(s)\ud s\right)\\ 
	=& \exp\begin{pmatrix}
     		0 & \overline{B(t)}+\int_{0}^{t}\overline{B(s)}\ud s \vspace{10pt}\\
     	B(t)+	\int_{0}^{t}{B(s)}\ud s & 0
     	\end{pmatrix} \vspace{10pt}\\[5pt]
     	=& \exp\begin{pmatrix}
     		0 &\mathrm{e}^{-i\Theta (t)}\left(r(t)+r'(t)-ir(t)\Theta'(t)\right) \vspace{10pt}\\
     		\mathrm{e}^{i\Theta (t)}\left(r(t)+r'(t)+ir(t)\Theta'(t)\right) & 0
     	\end{pmatrix}\vspace{10pt} \\[5pt]
     	=& \begin{pmatrix}  \cosh(\Lambda_1(t))&\frac{\mathrm{e}^{-i\Theta (t)}\sinh(\Lambda_1(t))}{\Lambda_1(t)}\left(r(t)+r'(t)-ir(t)\Theta'(t)\right)\vspace{10pt} \\ \frac{\mathrm{e}^{i\Theta (t)}\sinh(\Lambda_1(t))}{\Lambda_1(t)}\left(r(t)+r'(t)+ir(t)\Theta'(t)\right)  & \cosh(\Lambda_1(t))
     	\end{pmatrix}
     	\end{align*}
     and
     \begin{align*}
     	&\exp\left(D(t)\right)\cdot\exp\left(\int_{0}^{t}D(s)\ud s\right)\\
     	=&\begin{pmatrix}  \cosh(|B(t)|)&\sinh(|B(t)|)\frac{\overline{B(t)}}{|B(t)|}  \vspace{10pt}\\ \sinh(|B(t)|)\frac{{B(t)}}{|B(t)|} & \cosh(|B(t)|) \end{pmatrix}\begin{pmatrix}  \cosh(r(t))&\sinh(r(t))\mathrm{e}^{-i\Theta (t)}\vspace{10pt} \\ \sinh(r(t))\mathrm{e}^{i\Theta (t)} & \cosh(r(t)) \end{pmatrix}\vspace{10pt}\\[5pt]
     	=&\begin{pmatrix}
     		P_1(t) &P_2(t)\vspace{10pt}\\
     		P_3(t)&P_4(t)
     	\end{pmatrix}.
     \end{align*}
     According to the Lemma \ref{lem4}, we obtain that  \(D(t)\) commutes with \(\int_{0}^{t} D(s)\ud s\)  if and only if
     \begin{align*}\exp\left(D(t)\right)\cdot\exp\left(\int_{0}^{t}D(s)\ud s\right)=\exp\left(D(t)+\int_{0}^{t}D(s)\ud s\right).\end{align*}
     Consequently, we deduce that
     \begin{align*}
     	P_1(t)&=\cosh(|B(t)|)\cosh(r(t))+\frac{\mathrm{e}^{i\Theta (t)} \sinh(|B(t)|)\sinh(r(t))\overline{B(t)}}{|B(t)|}\vspace{10pt}\\[5pt]
     	&=\frac{\cosh\left(|B(t)|+r(t)\right)}{2}\left(1+\frac{\mathrm{e}^{i\Theta (t)}\overline{B(t)}}{|B(t)|}\right)+\frac{\cosh\left(|B(t)|-r(t)\right)}{2}\left(1-\frac{\mathrm{e}^{i\Theta (t)}\overline{B(t)}}{|B(t)|}\right)\vspace{10pt}\\[5pt]
     	&=\cosh(\Lambda_1(t)),\vspace{10pt}\\[5pt]
     	P_2(t)&=\cosh(|B(t)|)\sinh(r(t))\mathrm{e}^{-i\Theta (t)}+\frac{\sinh(|B(t)|)\cosh(r(t))\overline{B(t)}}{|B(t)|}\\[5pt]
     	&=\frac{\sinh(|B(t)|+r(t))}{2\mathrm{e}^{i\Theta (t)}}\left(1+\frac{\mathrm{e}^{i\Theta (t)}\overline{B(t)}}{|B(t)|}\right)+   	\frac{\sinh(r(t)-|B(t)|)}{2\mathrm{e}^{i\Theta (t)}}\left(1-\frac{\mathrm{e}^{i\Theta (t)}\overline{B(t)}}{|B(t)|}\right)\\[5pt]
     	&=\frac{\mathrm{e}^{-i\Theta (t)}\sinh(\Lambda_1(t))}{\Lambda_1(t)}\left(r(t)+r'(t)-ir(t)\Theta'(t)\right),\\[5pt]
     	P_3(t)&=\cosh(|B(t)|)\sinh(r(t))\mathrm{e}^{i\Theta (t)}+\frac{\sinh(|B(t)|)\cosh(r(t)){B(t)}}{|B(t)|}\\[5pt]
     	&=\frac{\sinh(|B(t)|+r(t))}{2\mathrm{e}^{-i\Theta (t)}}\left(1+\frac{\mathrm{e}^{-i\Theta (t)}{B(t)}}{|B(t)|}\right)+   	\frac{\sinh(r(t)-|B(t)|)}{2\mathrm{e}^{-i\Theta (t)}}\left(1-\frac{\mathrm{e}^{-i\Theta (t)}{B(t)}}{|B(t)|}\right)\\[5pt]
     	&=\frac{\mathrm{e}^{i\Theta (t)}\sinh(\Lambda_1(t))}{\Lambda_1(t)}\left(r(t)+r'(t)-ir(t)\Theta'(t)\right),\\[5pt]
     	P_4(t)&=\cosh(|B(t)|)\cosh(r(t))+\frac{\mathrm{e}^{-i\Theta (t)} \sinh(|B(t)|)\sinh(r(t)){B(t)}}{|B(t)|}\\[5pt]
     	&=\frac{\cosh\left(|B(t)|+r(t)\right)}{2}\left(1+\frac{\mathrm{e}^{-i\Theta (t)}{B(t)}}{|B(t)|}\right)+\frac{\cosh\left(|B(t)|-r(t)\right)}{2}\left(1-\frac{\mathrm{e}^{-i\Theta (t)}{B(t)}}{|B(t)|}\right)\\[5pt]
     	&=\cosh(\Lambda_1(t)).
     \end{align*}
     Moreover, it is worthwhile to note that if we claim that \(\overline{B(t)}=|B(t)|\mathrm{e}^{-i\Theta (t)}, \)  that is, \(\Theta'(t)=0,\) then 
     \begin{align*}
     P_1(t)=\cosh(r'(t)+r(t)) 
     \end{align*}
     and 
  \begin{align*}
  	\cosh(\Lambda_1(t))&=\cosh\left(\sqrt{\left(r(t)+r'(t)\right)^2+r^2(t)\Theta'^2(t)}\right)\\[5pt]
  	&=\cosh(r'(t)+r(t)).
  \end{align*} 
   Similarly, a straightforward calculation shows that if \(\Theta'(t)=0,\) that is \(\Theta(t)=k_0\in \mathbb{R},\) then we have
   \begin{align*}
   	P_2(t)&=\mathrm{e}^{-ik_0}\sinh(r'(t)+r(t)), \\ P_3(t)&=\mathrm{e}^{ik_0}\sinh(r'(t)+r(t)),\\ 
   	P_4(t)&=\cosh(r'(t)+r(t)).
   \end{align*}
   Therefore, we obtain that  \(\Theta(t)=k_0\) with \(k_0\in\mathbb{R}\) yields that   \(D(t)\) and \(\int_{0}^{t} D(s)\ud s\) commute. Then using the Lemma \ref{lem3}, we obtain the solutions
\begin{align*}
	\begin{cases}
		f(t,z)=\cosh(r(t))f_0(z)+\mathrm{e}^{-ik_0}\sinh(r(t))g_0(z),\vspace{10pt}\\
		g(t,z)=\cosh(r(t))g_0(z)+\mathrm{e}^{ik_0}\sinh(r(t))f_0(z).\\
	\end{cases}
\end{align*}
\end{proof}
\begin{remark}
We observe that Theorem \ref{thm1} constitutes a special case of Theorem \ref{lem2}. In fact, if we denote 
\begin{align*}
\alpha(t)=\cosh(r(t)),\quad \beta(t)=\sinh(r(t))\mathrm{e}^{-ik_0},\quad c=1, \quad\gamma=0,
\end{align*}
then \eqref{r1} holds.
\end{remark}

 \begin{remark}

If \(A(t)=\overline{B(t)},\) in light of \eqref{k3}, then we have \(\delta(z,\bar{z})=0\)  and \(\sigma(z,\bar{z})=0.\)
Assume that \(\mu=\kappa=0,\) then 
\begin{align*}
y_b\Gamma_{at}=y_a\Gamma_{bt}.
\end{align*}
 Moreover, from \eqref{1.15} we can deduce that 
 \begin{align*}
 	\begin{cases}
 		y_a=\displaystyle\frac{f-g-\overline{f}+\overline{g}}{2i},\vspace{10pt}\\
 		y_b=\displaystyle\frac{f-g+\overline{f}-\overline{g}}{2}.
 	\end{cases}
 \end{align*}
 By the definition of \(\Gamma\), we get
 \begin{align*}
 (f-g)(\theta_a-i\theta_b)-\overline{(f-g)}(\theta_a+i\theta_b)=0,
 \end{align*}
  which means that if \(f\neq g,\)  then we have 
 \begin{align*}
 \begin{cases}
 	\theta_a(t,x,y)=\theta_xx_a+\theta_yy_a=0,\\ \theta_b(t,x,y)=\theta_xx_b+ \theta_yy_b=0.
 \end{cases}
 \end{align*}
 From \eqref{1.1}, we can deduce that  \(\theta\) is a constant.
 Then  \eqref{1.1} becomes
 \begin{align*}
 	\begin{cases}
 		u_t+uu_x+vu_y+P_x=0,\\
 		v_t+uv_x+vv_y+P_y=c,
 	\end{cases}
 \end{align*}
 where \( c\in\mathbb{R}\) is a constant. Set \(c=0,\) we can see that \eqref{r1} provides an exact Lagrangian solution for the incompressible Euler equations under the inviscid assumption, describing fluid particle trajectories while satisfying mass and momentum conservation.
  \end{remark}
Next, we will discuss  the case of \(A(t)\neq\overline{B(t)}.\)  By \eqref{a1}, we find that the Jacoian  \(\mathcal{J}\) of the harmonic mapping \eqref{1.15} is time-independent in \(\Omega_0,\)  which allows us to apply Theorem \ref{lem3.1}.
    \begin{theorem}\label{theorem3.1}
    	Let \(\Omega_0\subset \mathbb{C}\) be a simply connected domain. Given the $C^1$-functions
     \(r : [0,\infty) \mapsto (0,\infty) \) and \(\phi : [0,\infty) \mapsto \mathbb{R} \).  If initial harmonic labelling map \(F_0+\overline{G_0}\) is  sense-preserving, and satisfies \(g_0=\lambda f_0\) where \(\lambda \in\mathbb{C}\), then the particle motion \(\eqref{2.13}\) of a fluid flow is given by
    	    \begin{align}
    		\begin{cases}
    		\displaystyle	f(t,z)=\sqrt{1-|\lambda|^2+r^2(t)}\exp\left(\displaystyle i\int_{0}^{t}\frac{cs+r^2(s)\phi'(s)+d}{1-|\lambda|^2+r^2(s)}~\ud s\right)f_0(z), \vspace{10pt}\\
    		\displaystyle	g(t,z)=r(t)\mathrm{e}^{i\phi(t)}f_0(z),
    		\end{cases}
    	\end{align}
    where \(~ (t,z)\in [0,\infty)\times\Omega_0,\)
 \(r^2(0)=|\lambda|^2 \) and \(c,d\in\mathbb{R}\) are  two  
 constants.
    \end{theorem}
    \begin{proof}
    	By \(\eqref{1.9}\), we observe that the Jacobian of the harmonic labelling map \(\eqref{1.15}\) is independent of time $t$. This allows us to apply the Theorem \(\ref{lem3.1}\). Therefore, we derive that
    	\begin{align*}
    	f(t,z)=\alpha(t)f_0(z),\quad g(t,z)=\beta(t)f_0(z),
        \end{align*}
    	where $\alpha(t)$ and \(\beta(t)\) satisfies
    	\begin{align}\label{3}
    	|\alpha(t)|^2-|\beta(t)|^2=1-|\lambda|^2.
    	\end{align}
    	In view of \(\eqref{2.13}\), we have
    
    \begin{align}
	\mathcal {L} f- \mathcal {L} \overline{g}&=\mathcal {L}\left(\alpha(t)f_0(z)\right)-\mathcal {L}\left(\beta(t)f_0(z)\right)\nonumber\\
	     &=\left(\alpha'(t)\overline{\alpha(t)}-\beta(t)\overline{\beta'(t)}\right)\left|f_0(z)\right|^2\nonumber\\
    		&=i\left(\delta(z,\bar{z})t+\sigma(z,\bar{z})\right).
    \end{align}
    Since \(|f_0(z)|^2\neq0\), then we deduce that
    \begin{align}
    	\alpha'(t)\overline{\alpha(t)}-\beta(t)\overline{\beta'(t)}=i\left(\hat{\delta}(z,\bar{z})t+\hat{\sigma}(z,\bar{z})\right),
    \end{align}
    where \(\hat{\delta}(z,\bar{z})=\frac{\delta(z,\bar{z})}{|f_0(z)|^2}\) and \(\hat{\sigma}(z,\bar{z})=\frac{\sigma(z,\bar{z})}{|f_0(z)|^2},\) 
    which means that
    \begin{align}
    	\alpha'(t)\overline{\alpha(t)}-\beta(t)\overline{\beta'(t)}=i(ct+d),
    \end{align}
     where two constants \(c,d\in\mathbb{R}\) are  independent of time $t$.
    Using the polar decompositions
    \begin{align*}
    \alpha(t)=R(t)\mathrm{e}^{i\Phi(t)},\quad \beta(t)=r(t)\mathrm{e}^{i\phi(t)}, 
    \end{align*}
    with \(R,r : [0,\infty) \mapsto (0,\infty) \) and \(\Phi, \phi : [0,\infty) \mapsto \mathbb{R} \) of class \(C^1\), then we obtain
    \begin{align}\label{3.5}
    	R(t)R'(t)-r(t)r'(t)+i\left(R^2(t)\Phi'(t)-r^2(t)\phi'(t)\right)=i(ct+d).
    \end{align}
    Moreover, differentiating \(\eqref{3}\) with respect to $t$ yields 
    \begin{align*}
    	\left(|\alpha(t)|^2-|\beta(t)|^2\right)_t&=(R^2(t)-r^2(t))_t\\
    	&=2(R(t)R'(t)-r(t)r'(t))\\
    	&=0.
    \end{align*}
    Then \(\eqref{3.5}\) becomes
    \begin{align*}
    R^2(t)\Phi'(t)-r^2(t)\phi'(t)=ct+d,
    \end{align*} 
    that is 
    \begin{align*}
    	\Phi'(t)=\frac{ct+r^2(t)\phi'(t)+d}{1-|\lambda|^2+r^2(t)}.
    \end{align*}
    Integrating it from 0 to $t$, we have
    \[\Phi(t)=\Phi(0)+\int_{0}^{t}\frac{cs+r^2(s)\phi'(s)+d}{1-|\lambda|^2+r^2(s)}~\ud s.
    \]
    Note that \(g_0(z)=\lambda f_0(z)\) implies \(\alpha(0)=1\) and \(\beta(0)=\lambda,\) then the motion of fluid flow \(\eqref{1.24}\) can be expressed as
   \begin{align}
   	\begin{cases}
   		\displaystyle	f(t,z)=\sqrt{1-|\lambda|^2+r^2(t)}\exp\left(\displaystyle i\int_{0}^{t}\frac{cs+r^2(s)\phi'(s)+d}{1-|\lambda|^2+r^2(s)}~\ud s\right)f_0(z), \vspace{10pt}\\
   		\displaystyle	g(t,z)=r(t)\mathrm{e}^{i\phi(t)}f_0(z).
   	\end{cases}
   \end{align}
     \end{proof}
\begin{remark}

According to the Theorem \ref{t1}, when \(A(t)\neq\overline{B(t)},\)  we get that  \(f(t,z)=\varrho(t)g(t,z)\) holds for all \(t\geq0.\) Therefore, here \(\lambda\) and \(\frac{\lambda\alpha(t)}{\beta(t)}\) are equivalent to \(\varrho(0)\)  and \(\varrho(t).\)
     Moreover, \(A(t)\) and \(B(t)\) satisfy the relations:
     \begin{align*}
     	\begin{cases}
     		\displaystyle A(t)=\frac{\lambda\alpha'(t)}{\beta(t)},\vspace{10pt}\\
     	\displaystyle	B(t)=\frac{\beta'(t)}{\alpha(t)\lambda},
     	\end{cases}
     \end{align*}
     where
     	\begin{align*}
     	|\alpha(t)|^2-|\beta(t)|^2=1-|\lambda|^2>0.
     \end{align*}
    \end{remark}
  \begin{proposition}
  	In the setting of Theorem \ref{theorem3.1}, for any instant \(t\), the harmonic mapping \eqref{1.15} is  sense-preserving in \(\Omega(t).\)
  \end{proposition}
  \begin{proof}
  	Applying the Theorem \ref{t2}, we have that \(G_0(z)\) is univalent. If the map \(F+\overline{G}\) is not univalent in \(\Omega_0\), then  there exists \(z_1\neq z_2\in \Omega(t)\) such that  \(F(t,z_1)+\overline{G(t,z_1)}=F(t,z_2)+\overline{G(t,z_2)}.\) Taking advantage of the relations \(F(t,z)=\alpha(t)F_0(z)\) and \(G(t,z)=\beta(t)G_0(z),\) we obtain 
  	\[F(t,z_1)+\overline{G(t,z_1)}=\lambda\alpha(t)G_0(z_1)+\overline{ \beta(t)G_0(z_1)}\] and \[F(t,z_2)+\overline{G(t,z_2)}=\lambda\alpha(t)G_0(z_2)+\overline{ \beta(t)G_0(z_2)}.\]
   Hence, we have
   \[\lambda\alpha(t)(G_0(z_1)-G_0(z_2))+\overline{\beta(t)\left(G_0(z_1)-G_0(z_2)\right)}=0.\]
   Since \(\alpha(t)\neq0\) and \(\beta(t)\neq0\), then we must have \(G_0(z_1)=G_0(z_2),\) which contradicts with the univalence of \(G_0(z).\) The sense-preserving property of the map \(F+\overline{G}\) is obtained by 
   \[  |F'(t,z)|^2-|G'(t,z)|^2=|F_0'(z)|^2-|G_0'(z)|^2>0.\]
  \end{proof}
     \begin{example}
     Let \(r(t)=r(0)=|\lambda|\), \(d=0,\) \(f_0(z)=kA\mathrm{e}^{ikz}\) and \(\phi=0\), where $A, k\in \mathbb{R}\setminus\{0\}$ are constants, we obtain
     	\[f(t,z)=kA\mathrm{e}^{i(ct+kz)},\quad g(t,z)=kA|\lambda|\mathrm{e}^{ikz}. \]
     	By means of the map \(\eqref{1.24}\), we deduce that
     	\begin{align*}
          x+iy
          &=kA\int_{0}^{z}\left(\mathrm{e}^{i(kw+ct)}+|\lambda|\mathrm{e}^{-ikw}\right)\ud w\\
          &=iA\left(\mathrm{e}^{ict}(1-\mathrm{e}^{ikz})+|\lambda|\mathrm{e}^{-ikz}+i|\lambda|\right).
     	\end{align*}
     	Furthermore, we  get the motion of a particle, which is given by
     	\begin{align*}
     		\begin{cases}
     			x(t,z)=A(\sin(kz+ct)-\sin(ct)+|\lambda|\sin(kz)-|\lambda|),\vspace{10pt}\\
     	    	y(t,z)=A(\cos(ct)-\cos(kz+ct)-|\lambda|\cos(kz)).
     		\end{cases}
     	\end{align*}
     	This shows that at every fixed \(z_0\in \Omega_0\) a fluid particle follows a circular trajectory.
     \end{example}
     
    \subsection{The linearly independent case }
    In this subsection, we find solutions \(f\neq0\) and \(g\neq0\) such that Eq.\eqref{2.13} holds and such
    that \(f_0\)  and \(g_0\) are linearly independent.
    \begin{theorem}
    	Let \(\Omega_0\subset \mathbb{C}\) be a simply connected domain. Given the $C^1$-functions
    	\(r : [0,\infty) \mapsto (0,\infty) \) and \(\psi : [0,\infty) \mapsto \mathbb{R} \). Let \(c_1,c_2,w,p,h,d_0\in\mathbb{R}\) be arbitrary constants. Assume that the initial harmonic labelling mapping \(F_0+\overline{G_0}\) is  sense-preserving, and \(f_0\), \(g_0\) are linearly independent.\\
    	   \textup{(i)}~~When \(c_1=0\) and \(c_2=0\), for all \(t\geq 0\) the particle motion \(\eqref{2.13}\) of a fluid flow  is described by
    	     \begin{align}\label{s1}
    	   	\begin{cases}
    	   		f(t,z)=\sqrt{1+r^2(t)}\exp\left(\displaystyle i\int_{0}^{t}\frac{hs+r^2(s)\psi'(s)+d_0}{1+r^2(s)}~\ud s\right)f_0(z)+r(t)\mathrm{e}^{\displaystyle i\psi(t)}g_0(z),\vspace{14pt} \\
    	   		g(t,z)=r(t)\mathrm{e}^{\displaystyle-i\psi(t)}f_0(z)+\sqrt{1+r^2(t)}\exp\left(\displaystyle-i\int_{0}^{t}\frac{hs+r^2(s)\psi'(s)+d_0}{1+r^2(s)}~\ud s\right)g_0(z).
    	   	\end{cases}
    	   \end{align}  
    	   \textup{(ii)}~~When \(c_1=0\) and \(c_2\neq0,\)  the particle motion \(\eqref{2.13}\) of a fluid flow, at any instant \(t\in\{\tau\geq 0: |w \tau+p|>|c_2|\}\), is given by
    	     \begin{align}\label{s2}
    	    	\begin{cases}
    	    		f(t,z)
    	    		&=\exp\left(\displaystyle i\int_{0}^{t}\frac{(2c_2h+w\psi'(s))s+(ws+p-c_2)\psi'(s)+2c_2d_0}{ws+p+c_2}~\ud s\right)\vspace{12pt}\\
    	    		&\quad\times\displaystyle \sqrt{\frac{wt+p+c_2}{2c_2}} f_0(z)+\displaystyle \sqrt{\frac{wt+p-c_2}{2c_2}}\mathrm{e}^{\displaystyle i(\psi(t)+c_2t)}g_0(z),\vspace{14pt} \\
    	    	   g(t,z)
    	    	   &=\displaystyle\sqrt{\frac{wt+p-c_2}{2c_2}}\mathrm{e}^{\displaystyle-i\psi(t)}f_0(z)+\displaystyle \sqrt{\frac{wt+p+c_2}{2c_2}}\mathrm{e}^{\displaystyle ic_2t}g_0(z)+\displaystyle \sqrt{\frac{wt+p+c_2}{2c_2}}\vspace{12pt}\\
    	    		&\quad\times\exp\left(\displaystyle -i\int_{0}^{t}\frac{(2c_2h+w\psi'(s))s+(ws+p-c_2)\psi'(s)+2c_2d_0}{ws+p+c_2}~\ud s\right)g_0(z).
    	    	\end{cases}
    	    \end{align}
    	    \textup{(iii)}~~When \(c_1\neq0\) and \(c_2=0,\) for \(|w|>|c_1|,\) then the particle motion \(\eqref{2.13}\) of a fluid flow follows from
    	     \begin{align}\label{s3}
    	    	\begin{cases}
    	    		f(t,z)&=\displaystyle \sqrt{\frac{w+c_1}{2c_1}}\exp\left(\displaystyle i\frac{c_1ht^2+2d_0 c_1 +(w-c_1)(\psi(t)-\psi(0))}{w+c_1}\right)f_0(z)\vspace{10pt}\\
    	    		&\quad+\displaystyle \sqrt{\frac{w-c_1}{2c_1}}\mathrm{e}^{\displaystyle i(c_1t^2+\psi(t))}g_0(z), \vspace{12pt}\\
    	    		g(t,z)&=\displaystyle \sqrt{\frac{w-c_1}{2c_1}}\mathrm{e}^{\displaystyle -i\psi(t)}f_0(z)\vspace{10pt}\\
    	    		&\quad\displaystyle+\sqrt{\frac{w+c_1}{2c_1}}\exp\left(\displaystyle ic_1t^2-i\frac{c_1ht^2+2d_0 c_1 +(w-c_1)(\psi(t)-\psi(0))}{w+c_1}\right)g_0(z).
    	    	\end{cases}
    	    \end{align}\\
    	     \textup{(iv)}~~When \(c_1\neq0\) and \(c_2\neq0,\) the particle motion \(\eqref{2.13}\) of a fluid flow,  at any instant \(t\in\left\{\tau\geq 0: \left|\displaystyle\frac{w \tau+p}{c_1\tau+c_2}\right|>1 \right\},\) is described by
\begin{align}\label{s4}
	\begin{cases}
		f(t,z) &= \displaystyle\sqrt{\frac{(w + c_1)t + p + c_2}{2(c_1t + c_2)}} \,
		\exp\left(\displaystyle i \int_0^t \frac{(w - c_1)s + p - c_2)\psi'(s)}{(c_1 + w)s + p + c_2} \, \ud s \right) f_0(z) \\[20pt]
		&\quad\times\displaystyle \sqrt{\frac{(w + c_1)t + p + c_2}{2(c_1t + c_2)}} \,
		\exp\left(\displaystyle- i \int_0^t \frac{2(hs + d_0)(c_1s + c_2)}{(c_1 + w)s + p + c_2} \, \ud s \right) f_0(z) \\[20pt]
		&\quad + \displaystyle\sqrt{\frac{(w - c_1)t + p - c_2}{2(c_1t + c_2)}} \,
		\mathrm{e}^{\displaystyle i(c_1t^2 + c_2t + \psi(t))} g_0(z), \\[24pt]
		g(t,z) &= \displaystyle\sqrt{\frac{(w - c_1)t + p - c_2}{2(c_1t + c_2)}} \,
		\mathrm{e}^{\displaystyle-i\psi(t)} f_0(z) +\displaystyle \sqrt{\frac{(w + c_1)t + p + c_2}{2(c_1t + c_2)}} \,
		\mathrm{e}^{\displaystyle i(c_1t^2 + c_2t)} g_0(z) \\[20pt]
		&\quad  \displaystyle\times\sqrt{\frac{(w + c_1)t + p + c_2}{2(c_1t + c_2)}} \,
		\exp\left(\displaystyle-i \int_0^t \frac{(w - c_1)s + p - c_2)\psi'(s)}{(c_1 + w)s + p + c_2} \, \ud s \right) g_0(z) \\[20pt]
		&\quad\times \displaystyle\sqrt{\frac{(w + c_1)t + p + c_2}{2(c_1t + c_2)}} \,
		\exp\left(\displaystyle i \int_0^t \frac{2(hs + d_0)(c_1s + c_2)}{(c_1 + w)s + p + c_2} \, \ud s \right) g_0(z).
	\end{cases}
\end{align}

    \end{theorem}
    \begin{proof}

    According to Theorem \(\ref{lem2}\) and  taking \(c=1\),  there exist \(C^1\) functions \(\alpha, \beta : [0,\infty) \mapsto \mathbb C,\) where $|\alpha(t)|^2=1+|\beta(t)|^2$, for all $t\geq0$ and \( \gamma : [0,\infty) \mapsto \mathbb{R} \) such that
    \begin{align}\label{3.9}
    	\begin{cases}
    		f(t,z)=\alpha(t)f_0(z)+\mathrm{e}^{i\gamma(t)}\beta(t)g_0(z),\\[10pt]
    		g(t,z)=\overline{\beta(t)}f_0(z)+\mathrm{e}^{i\gamma(t)}\overline{\alpha(t)}g_0(z).
    	\end{cases}
    \end{align}
    Recall that 
    \begin{align}\label{T}
    	\mathcal {L} f- \mathcal {L} \overline{g}=i(\delta(z,\bar{z})t+\sigma(z,\bar{z})).
    \end{align}
    Set 
    \begin{align}\label{w}
    	\zeta(t)=\mathrm{e}^{i\gamma(t)}\beta(t),\quad \quad\eta(t)=\mathrm{e}^{i\gamma(t)}\overline{\alpha(t)}.
    \end{align}
    Using  \(\eqref{R}\), we obtain
    \begin{align*}
    	\mathcal {L}f&=\mathcal {L}\big(\alpha(t)f_0(z)+\zeta(t)
    	g_0(z)\big)\\[5pt]
    	&=\mathcal {L}\big(\alpha(t)f_0(z)\big)
    	+\mathcal {L}\big(\zeta(t)g_0(z)\big)+\overline{\alpha(t)}\zeta'(t)\overline{f_0(z)}g_0(z)+
    	\overline{\zeta(t)}\alpha'(t)\overline{g_0(z)}f_0(z)\\[5pt]
    	&=|\alpha(t)|^2\mathcal {L}f_0(z)+|f_0(z)|^2\mathcal {L}\alpha(t)+
    	|\zeta(t)|^2\mathcal {L}g_0(z)+|g_0(z)|^2\mathcal {L}\zeta(t)\\[5pt]
    	&\quad+\overline{\alpha(t)}\zeta'(t)\overline{f_0(z)}g_0(z)+
    	\overline{\zeta(t)}\alpha'(t)\overline{g_0(z)}f_0(z)\\[5pt]
    	&=|f_0(z)|^2\mathcal {L}\alpha(t)+|g_0(z)|^2\mathcal {L}\zeta(t)+\overline{\alpha(t)}\zeta'(t)\overline{f_0(z)}g_0(z)+
    	\overline{\zeta(t)}\alpha'(t)\overline{g_0(z)}f_0(z)\\[5pt]
    	&=|f_0(z)|^2\alpha'(t)\overline{\alpha(t)}+|g_0(z)|^2\zeta'(t)\overline{\zeta(t)}+\overline{\alpha(t)}\zeta'(t)\overline{f_0(z)}g_0(z)+
    	\overline{\zeta(t)}\alpha'(t)\overline{g_0(z)}f_0(z).
    \end{align*}
    Similarly, we have
    \begin{align*}
    		\mathcal {L}g&=\mathcal {L}\left(\overline{\beta(t)}f_0(z)+\eta(t)g_0(z)\right)\\[5pt]
    	&=\mathcal {L}\big(\overline{\beta(t)}f_0(z)\big)
    	+\mathcal {L}\big(\eta(t)g_0(z)\big)+\beta(t)\eta'(t)\overline{f_0(z)}g_0(z)+\overline{\eta(t)}\overline{\beta'(t)}\overline{g_0(z)}f_0(z)\\[5pt]
    	&=|\beta(t)|^2\mathcal {L}f_0(z)+|f_0(z)|^2\mathcal {L}\overline{\beta(t)}+
    	|\eta(t)|^2\mathcal {L}g_0(z)+|g_0(z)|^2\mathcal {L}\eta(t)\\[5pt]
    	&\quad+\beta(t)\eta'(t)\overline{f_0(z)}g_0(z)+\overline{\eta(t)}\overline{\beta'(t)}\overline{g_0(z)}f_0(z)\\[5pt]
    	&=|f_0(z)|^2\mathcal {L}\overline{\beta(t)}+|g_0(z)|^2\mathcal {L}\eta(t)+
    	\beta(t)\eta'(t)\overline{f_0(z)}g_0(z)+\overline{\eta(t)}\overline{\beta'(t)}\overline{g_0(z)}f_0(z)\\[5pt]
    	&=|f_0(z)|^2\overline{\beta'(t)}\beta(t)+|g_0(z)|^2\eta'(t)\overline{\eta(t)}+
    	\beta(t)\eta'(t)\overline{f_0(z)}g_0(z)+
    	\overline{\eta(t)}\overline{\beta'(t)}\overline{g_0(z)}f_0(z),
    \end{align*}
    which implies that
    \begin{align*}
	    \mathcal {L} \overline{g}=|f_0(z)|^2\beta'(t)\overline{\beta(t)}+|g_0(z)|^2\overline{\eta'(t)}\eta(t)+
    \overline{\beta(t)\eta'(t)}\overline{g_0(z)}f_0(z)+
    \eta(t)\beta'(t)g_0(z)\overline{f_0(z)}.
    \end{align*}
    Furthermore, in light of \(\eqref{w}\), we get
    \begin{align*}
    	\mathcal {L} f- \mathcal {L} \overline{g}&=|f_0(z)|^2\left(\alpha'(t)\overline{\alpha(t)}-\beta'(t)\overline{\beta(t)}\right)+|g_0(z)|^2\left(\zeta'(t)\overline{\zeta(t)}-\overline{\eta'(t)}\eta(t)\right)\\[5pt]
    	&\quad+
    	\overline{f_0(z)}g_0(z)\left(\overline{\alpha(t)}\zeta'(t)-\eta(t)\beta'(t)\right)
    	+\overline{g_0(z)}f_0(z)\left(\overline{\zeta(t)}\alpha'(t)-\overline{\beta(t)\eta'(t)}\right)\\[5pt]
    	&=\left(|f_0(z)|^2-|g_0(z)|^2\right)\left(\alpha'(t)\overline{\alpha(t)}-
    	\beta'(t)\overline{\beta(t)}\right)\\[5pt]
    	&\quad+2i\text{Re}\left(\gamma'(t)\beta(t)\overline{\alpha(t)}\mathrm{e}^{i\gamma(t)}
    	\overline{f_0(z)}g_0(z)\right)+i\gamma'(t)(|\alpha(t)|^2+|
    	\beta(t)|^2)|g_0(z)|^2\\[5pt]
    	&=i(\delta(z,\bar{z})t+\sigma(z,\bar{z})).
    \end{align*}
    Since \(|f_0(z)|^2-|g_0(z)|^2>0,\) then we deduce that
    \begin{align}\label{V}
    	&\alpha'(t)\overline{\alpha(t)}-
    	\beta'(t)\overline{\beta(t)}
    	+\frac{i}{|f_0(z)|^2-|g_0(z)|^2}\nonumber\\[5pt]
    	&\quad\times\left(2\text{Re}\left(\gamma'(t)\beta(t)\overline{\alpha(t)}\mathrm{e}^{i\gamma(t)}
    	\overline{f_0(z)}g_0(z)\right)+\gamma'(t)(|\alpha(t)|^2+|
    	\beta(t)|^2)|g_0(z)|^2\right)\nonumber\\[5pt]
    	=&i(\tilde{\delta}(z,\bar{z})t+\tilde{\sigma}(z,\bar{z})),
    \end{align}
    where \(\tilde{\delta}(z,\bar{z})=\frac{\delta(z,\bar{z})}{|f_0(z)|^2-|g_0(z)|^2}\) and
     \(\tilde{\sigma}(z,\bar{z})=\frac{\sigma(z,\bar{z})}{|f_0(z)|^2-|g_0(z)|^2}.\) 
    Since we assume that the harmonic mapping $F_0+\overline{G_0}$ is sense-preserving, then we get \(|F_0'|>0\) in \(\Omega_0\), and the  dilatation of
     \(F_0+\overline{G_0}\), \(q(z)=\frac{G'_0(z)}{F'_0(z)}=\frac{g_0(z)}{f_0(z)}\), is an analytic function in \(\Omega_0\) with \(|q(z)|<1.\) \(q(z)\)
     is not a constant because  we assume that  \(f_0(z)\) and \(g_0(z)\) are linearly independent.  Hence there is an open disk \(M\subset \Omega_0\) such that \(q'(z)\neq0\) for all \(z\in M\). Then the equation \(\eqref{V}\)  can be rewritten as
    \begin{align}\label{L1}
    	&\alpha'(t)\overline{\alpha(t)}-
    	\beta'(t)\overline{\beta(t)}
    	+2i\frac{\text{Re}\left(\gamma'(t)\beta(t)\overline{\alpha(t)}\mathrm{e}^{i\gamma(t)}
    		q(z)\right)}{1-|q(z)|^2}\nonumber
    		+i\frac{\gamma'(t)(|\alpha(t)|^2+|
    			\beta(t)|^2)|q(z)|^2}{1-|q(z)|^2}\\
    			&=i(\tilde{\delta}(z,\bar{z})t+\tilde{\sigma}(z,\bar{z})).
    \end{align}
    Differentiating \(\eqref{L1}\) with respect to \(\bar{z}\) yields
       \begin{align*}
       	 &i\frac{\partial }{\partial \bar{z}}\left(\frac{\gamma'(t)\beta(t)\overline{\alpha(t)}\mathrm{e}^{i\gamma(t)}q(z)+\gamma'(t)\overline{\beta(t)}\alpha(t)\mathrm{e}^{-i\gamma(t)}
       		\overline{q(z)}}{1-|q(z)|^2}\right)\nonumber+i\gamma'(t)(|\alpha(t)|^2+|
       		\beta(t)|^2)\\
       		&\quad\times\frac{\partial}{\partial \bar{z}}\left(\frac{|q(z)|^2}{1-|q(z)|^2}\right)\nonumber\\[5pt]
       		=&i\frac{\gamma'(t)\overline{\beta(t)}\alpha(t)\mathrm{e}^{-i\gamma(t)}
	       		q'(z)q^2(z)+\gamma'(t)\beta(t)\overline{\alpha(t)}\mathrm{e}^{i\gamma(t)}
       	q'(z)}{\left(1-|q(z)|^2\right)^2}\\[5pt]
       		&\quad+i\frac{\gamma'(t)(|\alpha(t)|^2+|
       			\beta(t)|^2)q(z)q'(z)}{\left(1-|q(z)|^2\right)^2}\\[5pt]
       		=&i(\tilde{\delta}_{\bar{z}}(z,\bar{z})t+\tilde{\sigma}_{\bar{z}}(z,\bar{z})).
    \end{align*}
    Since \(q'(z)\neq0\), then 
    \begin{align}\label{M1}
    	&\gamma'(t)\overline{\beta(t)}\alpha(t)\mathrm{e}^{-i\gamma(t)}
    		q^2(z)+\gamma'(t)\beta(t)\overline{\alpha(t)}\mathrm{e}^{i\gamma(t)}+\gamma'(t)(|\alpha(t)|^2+|
    		\beta(t)|^2)q(z)\nonumber\\[5pt]
    		=&\frac{\left(1-|q(z)|^2\right)^2}{q'(z)}
    		(\tilde{\delta}_{\bar{z}}(z,\bar{z})t+\tilde{\sigma}_{\bar{z}}(z,\bar{z})).
    \end{align}
    Taking derivatives with respect to \(z\) in \(\eqref{M1}\), we obtain that
    \begin{align*}
    	\gamma'(t)\overline{\beta(t)}\alpha(t)\mathrm{e}^{-i\gamma(t)}
    	\frac{q(z)}{q'(z)}+\gamma'(t)(|\alpha(t)|^2+|
    	\beta(t)|^2)=\frac{\nu_1(z,\bar{z})t+\nu_2(z,\bar{z})}{q'(z)},
    \end{align*}
    which means that 
    \begin{align}\label{3.15}
    	\gamma'(t)\overline{\beta(t)}\alpha(t)\mathrm{e}^{-i\gamma(t)}=kt+m,\quad k,m\in\mathbb{C},
    \end{align}
    and
    \begin{align}\label{3.16}
    	\gamma'(t)(|\alpha(t)|^2+|
    	\beta(t)|^2)=wt+p,\quad w,p\in\mathbb R.
    	\end{align}
    Since \(|\alpha(t)|^2=1+|\beta(t)|^2\), by \(\eqref{L1}\), then \(\alpha'(t)\overline{\alpha(t)}-\beta'(t)\overline{\beta(t)}\) must be a purely imaginary, that is
    \begin{align}\label{3.17}
    	\alpha'(t)\overline{\alpha(t)}-\beta'(t)\overline{\beta(t)}=i(ht+d_0), \quad h,d_0\in\mathbb{R}.
    \end{align}
    Note that from \(|\alpha(t)|^2=1+|\beta(t)|^2\) and \(\eqref{3.15}\), we find
    \begin{align}\label{3.18}
    	\begin{cases}
    		\gamma'(t)|\alpha(t)|^2=\displaystyle \frac{wt+p+\gamma'(t)}{2},\\[10pt]
    		\gamma'(t)|\beta(t)|^2=\displaystyle \frac{wt+p-\gamma'(t)}{2}.
    	\end{cases}
    \end{align}
    Furthermore, in light of \(\eqref{3.15}\)  and \(\eqref{3.18}\), we obtain
    \begin{align*}
    (\gamma'(t))^2|\alpha(t)|^2|\beta(t)|^2&=\frac{(wt+p)^2- (\gamma'(t))^2}{4}\\[5pt]
    &=\frac{w^2t^2++2wpt+p^2-(\gamma'(t))^2}{4}\\[5pt]
    &=|k|^2t^2+t(\overline{k}m+k\overline{m})+|m|^2,
    \end{align*}
    which shows that \(\gamma'(t)\) satisfies
    \(\gamma'(t)=c_1t+c_2\). 
     Moreover, from \(\eqref{3.9}\) we deduce that \(\gamma(0)=0\). 
     Hence we derive that \(\gamma(t)=c_1t^2+c_2t.\) In addition, by \(\eqref{3.16}\), we obtain 
     \begin{align}\label{3.19}
     	(c_1t+c_2)(1+2|\beta(t)|^2)=wt+p,
     \end{align}
     and
       \begin{align}\label{3.20}
     	(c_1t+c_2)(2|\alpha(t)|^2-1)=wt+p.
     \end{align}
    Next, we will discuss the following four types of solutions.\\
    $\mathbf{Case ~1}\quad c_1=0 \text{ and } c_2=0.$ Then we deduce \(\gamma=0\) and 
    \begin{align}
    		f(t,z)=\alpha(t)f_0(z)+\beta(t)g_0(z),\quad
    	g(t,z)=\overline{\beta(t)}f_0(z)+\overline{\alpha(t)}g_0(z).
    \end{align}
     Recall that the polar decompositions
    \begin{align}\label{3.22}
    \alpha(t)=R(t)\mathrm{e}^{i\Psi(t)},\quad \beta(t)=r(t)\mathrm{e}^{i\psi(t)}, 
    \end{align}
    with \(R, r : [0,\infty) \mapsto (0,\infty) \) and \(\Psi, \psi : [0,\infty) \mapsto \mathbb{R} \) of class \(C^1\). Here the modules of \(\alpha(t)\) and \(\beta(t)\) satisfy \(|\alpha(t)|^2-|\beta(t)|^2=1.\)  Using  a similar approach to Theorem \(\ref{theorem3.1}
    \), we further obtain that
    \begin{align*}
    	\Psi(t)=\Psi(0)+\int_{0}^{t}\frac{hs+r^2(s)\psi'(s)+d_0}{1+r^2(s)}~\ud s.
    	 \end{align*}
    Moreover, from \(\eqref{3.9}\), it is not difficult to observe that
    \begin{align*}
    	\alpha(0)=1, \quad \beta(0)=0.
    \end{align*}
    Therefore, we get the solutions:
    \begin{align*}
 	\begin{cases}
 		f(t,z)=\sqrt{1+r^2(t)}\exp\left(\displaystyle i\int_{0}^{t}\frac{hs+r^2(s)\psi'(s)+d_0}{1+r^2(s)}~\ud s\right)f_0(z)+r(t)\mathrm{e}^{\displaystyle i\psi(t)}g_0(z), \\[16pt]
 		g(t,z)=r(t)\mathrm{e}^{\displaystyle-i\psi(t)}f_0(z)+\sqrt{1+r^2(t)}\exp\left(\displaystyle-i\int_{0}^{t}\frac{hs+r^2(s)\psi'(s)+d_0}{1+r^2(s)}~\ud s\right)g_0(z).
 	\end{cases}
 \end{align*}
 $\mathbf{Case ~2}\quad c_1=0 \text{ and } c_2\neq0.$
  Then \(\eqref{3.19}\) and \(\eqref{3.20}\) become
  \begin{align}\label{3.23}
  	c_2(1+2|\beta(t)|^2)=wt+p,
  \end{align}
  and
  \begin{align}\label{3.24}
  	c_2(2|\alpha(t)|^2-1)=wt+p.
  \end{align}
  This shows that 
  \begin{align*}
  	\begin{cases}
  		R^2(t)=\displaystyle\frac{wt+p+c_2}{2c_2},\\[10pt]
  		r^2(t)=\displaystyle \frac{wt+p-c_2}{2c_2}.
  	\end{cases}
  \end{align*}
  Due to \(R>0\) and \(r>0\), then we can see that the function \(\vartheta(t)=wt+p\) must have \(|\vartheta|>|c_2|.\)
  Similarly, \(\eqref{3.17}\) becomes
  \begin{align*}
	&R(t)R'(t)-r(t)r'(t)+i\left(R^2(t)\Phi'(t)-r^2(t)\phi'(t)\right)\\[5pt]&=i\left(R^2(t)\Psi'(t)-r^2(t)\psi'(t)\right)\\[5pt]
	&=i\frac{(wt+p+c_2)\Phi'(t)-(wt+p-c_2)\psi'(t)}{2c_2}\\[5pt]
	&=
	i(ht+d_0).
\end{align*}
Therefore, we deduce that
\begin{align*}
	\Psi'(t)=\frac{(2c_2h+w\psi'(t))t+(wt+p-c_2)\psi'(t)+2c_2d_0}{wt+p+c_2}.
\end{align*}
Integrating it from 0 to \(t\), we have
\begin{align*}
	\Psi(t)&=\Psi(0)+\int_{0}^{t}\frac{(2c_2h+w\psi'(s))s+(ws+p-c_2)\psi'(s)+2c_2d_0}{ws+p+c_2}~\ud s.
\end{align*}
Then  we get the  solutions \eqref{s2}.\\
$\mathbf{Case ~3}\quad c_1\neq0 \text{ and } c_2=0.$
  Then \(\eqref{3.19}\) and \(\eqref{3.20}\) become
  \begin{align}\label{3.25}
  	c_1t(1+2|\beta(t)|^2)=wt+p,
  \end{align}
  and
  \begin{align}\label{3.26}
  	c_1t(2|\alpha(t)|^2-1)=wt+p.
  \end{align}
From  \(\eqref{3.25}\) and \(\eqref{3.26}\) we get
   \begin{align*}
 \begin{cases}
 		R^2(t)=\displaystyle \frac{w+c_1}{2c_1},\\[10pt]
r^2(t)=\displaystyle \frac{w-c_1}{2c_1},
 \end{cases}
 \end{align*}
 Moreover, \(R>0\) and \(r>0\) lead to  \(|w|>|c_1|.\)
 Similarly, \(\eqref{3.17}\) becomes
 \begin{align*}
 	&R(t)R'(t)-r(t)r'(t)+i\left(R^2(t)\Phi'(t)-r^2(t)\phi'(t)\right)\\[5pt]&=i\left(R^2(t)\Psi'(t)-r^2(t)\psi'(t)\right)\\[5pt]
 	&=i\frac{(w+c_1)\Phi'(t)-(w-c_1)\psi'(t)}{2c_1}\\[5pt]
 	&=
 	i(ht+d_0).
 \end{align*}
 Therefore, we deduce 
 \begin{align*}
 	\Psi'(t)=\frac{2c_1(ht+d_0)+(w-c_1)\psi'(t)}{w+c_1}.
 \end{align*}
 Integrating it from 0 to \(t\), we have
 \begin{align*}
 	\Psi(t)&=\Psi(0)+\frac{1}{w+c_1}\left(c_1ht^2+2d_0 c_1 +(w-c_1)(\psi(t)-\psi(0))\right).
 \end{align*}
 Then we obtain the  solutions  \eqref{s3}.\\
 $\mathbf{Case ~4}\quad c_1\neq0 \text{ and } c_2\neq0.$
  Then we obtain
  \begin{align}
  \begin{cases}
  		(c_1t+c_2)(1+2r^2(t))=wt+p,\\[5pt]
  (c_1t+c_2)(2R^2(t)-1)=wt+p,
  \end{cases}
  \end{align}
  which yields that
    \begin{align}\label{M}
    \begin{cases}
    	R^2(t)=\displaystyle \frac{(w+c_1)t+p+c_2}{2(c_1t+c_2)},\\[10pt]
      r^2(t)=\displaystyle \frac{(w-c_1)t+p-c_2}{2(c_1t+c_2)}.
  \end{cases}
    \end{align}
    Since \(R(t),r(t)\)  are positive and continuous function in \([0,\infty),\) by \(\eqref{M}\) we can know that the function \(\varsigma(t)=\frac{wt+p}{c_1t+c_2}\) must satisfy \(|\varsigma|>1.\)
    The equation \(\eqref{3.17}\) becomes 
     \begin{align*}
    	&R(t)R'(t)-r(t)r'(t)+i\left(R^2(t)\Phi'(t)-r^2(t)\phi'(t)\right)\\[5pt]&=i\left(R^2(t)\Psi'(t)-r^2(t)\psi'(t)\right)\\[5pt]
    	&=i\left(\frac{1}{2}\left(1+\frac{wt+p}{c_1t+c_2}\right)\Psi'(t)-\frac{1}{2}\left(\frac{wt+p}{c_1t+c_2}-1\right)\psi'(t)\right)\\[5pt]
    	&=
    	i(ht+d_0).
    \end{align*}
    Then we have
    \begin{align*}
    	\Psi'(t)=\frac{\left((w-c_1)t+p-c_2\right)\psi'(t)+2(ht+d_0)(c_1t+c_2)}{(c_1+w)t+p+c_2},
    \end{align*}
    which means that
    \begin{align*}
    	\Psi(t)=\Psi(0)+\int_{0}^{t}\frac{\left((w-c_1)s+p-c_2\right)\psi'(s)-2(hs+d_0)(c_1s+c_2)}{(c_1+w)s+p+c_2}~\ud s.
    \end{align*}
    Then we obtain the  solutions \(\eqref{s4}.\)
        \end{proof}

        \begin{example}
        	Set \(r(t)=r_0>0\), \( c_1=c_2=d_0=0,\) \(\psi(t)=\frac{1+r^2_0}{r^2_0}t,\) \(h=1+r^2_0,\) \(f_0(z)=\frac{r^2_0}{\sqrt{1+r^2_0}}\mathrm{e}^{ik_1z}\) and \(g_0(z)=\frac{1}{r_0}\mathrm{e}^{ik_2z}\), where $ k_1\neq k_2,r_0\in \mathbb{R}\setminus\{0\}$ are constants. Then we obtain
        	\[f(t,z)=\mathrm{ e}^{\displaystyle i(t^2+t+k_1z)}+\exp\left(\displaystyle i\left(\frac{1+r^2_0}{r^2_0}t+k_2z\right)\right)\] and 
        	\[g(t,z)=\displaystyle\frac{r^3_0}{\sqrt{1+r^2_0}}\exp\left(\displaystyle i\left(k_1z-\frac{1+r^2_0}{r^2_0}t\right)\right)+\frac{\sqrt{1+r^2_0}}{r_0}\exp\left(\displaystyle i\left(k_2z-t^2-t\right)\right).\]
        \end{example}
        \begin{example}
        	Set \(w=c_2\neq0\), \( c_1=p=h=d_0=0,\) \(\psi(t)=t,\)  \(f_0(z)=A_1\mathrm{e}^{ik_1z}\) and \(g_0(z)=A_2\mathrm{e}^{ik_2z}\), where $w,c_2, k_1\neq k_2,A_1\neq A_2$ are nonzero constants. Then for every instant \(t\geq1\) we get
        	\[f(t,z)=A_1\sqrt{\frac{t+1}{2}}\mathrm{ e}^{\displaystyle i(2t-3\ln(t+1) +k_1z)}+A_2\sqrt{\frac{t-1}{2}}\mathrm{ e}^{\displaystyle i\left((1+c_2)t+k_2z\right)}\] and 
        	\[g(t,z)=A_1\sqrt{\frac{t-1}{2}}\mathrm{ e}^{\displaystyle i\left(k_1z-t\right)}+A_2\sqrt{\frac{t+1}{2}}\mathrm{ e}^{\displaystyle i\left(k_2z+(c_2-2)t+3\ln(t+1)\right)}.\]
        \end{example}
        \begin{example}
        	Set \(w=2c_1\neq0\), \( c_2=d_0=0,\) \(\psi(t)=t,\)  \(f_0(z)=A_1\mathrm{e}^{ik_1z}\) and \(g_0(z)=A_2\mathrm{e}^{ik_2z}\), where $w,c_1, k_1\neq k_2,A_1\neq A_2$ are nonzero constants. Then we get
        	\[f(t,z)=\frac{\sqrt{6}}{2}A_1\exp\left(\displaystyle i\left(\frac{1}{3}(ht^2+t)+k_1z\right)\right)+\frac{\sqrt{2}}{2}A_2\mathrm{ e}^{\displaystyle i\left(c_1t^2+t+k_2z\right)}\] and 
        	\[g(t,z)=\frac{\sqrt{2}}{2}A_1\mathrm{ e}^{ \displaystyle i\left(k_1z-t\right)}+\frac{\sqrt{6}}{2}A_2\exp\left(\displaystyle i\left(c_1t^2+k_2z-\frac{ht^2+t}{3}\right)\right).\]
        \end{example}
        \begin{example}
        	Set \(w=2c_1\neq0\), \( p=2c_2\neq0,\) \(d_0=0,\) \(\psi(t)=t,\)  \(f_0(z)=A_1\mathrm{e}^{ik_1z}\) and \(g_0(z)=A_2\mathrm{e}^{ik_2z}\), where $w,c_1,p,c_2, k_1\neq k_2,A_1\neq A_2$ are nonzero constants. Then we have
        	\[f(t,z)=\frac{\sqrt{6}}{2}A_1\exp\left(\displaystyle i\left(\frac{1}{3}(ht^2+t)+k_1z\right)\right)+\frac{\sqrt{2}}{2}A_2\exp\left(\displaystyle i\left(c_1t^2+(c_2+1)t+k_2z\right)\right)\] and 
        	\[g(t,z)=\frac{\sqrt{2}}{2}A_1\mathrm{ e}^{ \displaystyle i\left(k_1z-t\right)}+\frac{\sqrt{6}}{2}A_2\exp\left(\displaystyle i\left(c_1t^2+c_2t+k_2z-\frac{ht^2+t}{3}\right)\right).\]
        \end{example} 
       \section{A General Class of Solutions}\label{sect5} 
       In this section, we  derive the general results.

       \subsection{The linearly dependent case }
       In this subsection, we begin by seeking the  solutions 
       \(f\neq0\) and \(g\neq0\)
   that satisfy the governing Eq. \eqref{4.1}, under the additional condition that the initial harmonic (sense-preserving) labeling map 
       \(F_0+\overline{G_0}\)
       is such that 
       \(F'_0\)
       and 
     \(G'_0\)
       are linearly dependent.
       \begin{theorem}\label{theorem4.1}
       	Let \(\Omega_0\subset \mathbb{C}\) be a simply connected domain. Given the $C^1$-functions
       	\(r : [0,\infty) \mapsto (0,\infty) ,\) \(\Xi : [0,\infty) \mapsto \mathbb{C}\setminus\{0\} \) and \(\phi : [0,\infty) \mapsto \mathbb{R} \).  If initial harmonic labelling map \(F_0+\overline{G_0}\) is univalent and sense preserving, and satisfies \(g_0=\lambda f_0,\) where \(\lambda \in\mathbb{C}\), then the particle motion \(\eqref{4.1}\) of a fluid flow is described by
       	\begin{align}\label{ee}
       		\begin{cases}
       			\displaystyle	f(t,z)=\sqrt{1-|\lambda|^2+r^2(t)}\exp\left(\displaystyle i\int_{0}^{t}\frac{r^2(s)\phi'(s)-i\Xi(s)}{1-|\lambda|^2+r^2(s)}~\ud s\right)f_0(z), \\[16pt]
       			\displaystyle	g(t,z)=r(t)\mathrm{e}^{i\phi(t)}f_0(z),
       		\end{cases}
       	\end{align}
       	where \((t,z)\in [0,\infty)\times\Omega_0.\)
       \end{theorem}
       \begin{proof}
       	Since the poof of this theorem is very similar to Theorem \ref{theorem3.1}, we only present the key steps. According to Theorem \ref{thm4}, \(g_0(z)=\lambda f_0(z)\) means that \[\mathcal{K}(t,z,\overline{z})=\Xi(t)\mathcal{Q}(z),\]
       	where \[\Xi(t)=C_1(t)+|\lambda|^2C_2(t)+
       	\overline{\lambda}C_3(t)+\lambda C_4(t),\]
       	and \[\mathcal{Q}(z)=-i|f_0(z)|^2.\]
       	Applying the Theorem \(\ref{lem3.1}\), we get that
       	\begin{align*}
       		f(t,z)=\alpha(t)f_0(z),\quad g(t,z)=\beta(t)f_0(z),
       	\end{align*}
       	where $\alpha(t)$ and \(\beta(t)\) satisfy the relation:
       	\begin{align*}
       		|\alpha(t)|^2-|\beta(t)|^2=1-|\lambda|^2.
       	\end{align*}
       	Recall that the polar decompositions
       	\[\alpha(t)=R(t)\mathrm{e}^{i\Phi(t)},\quad \beta(t)=r(t)\mathrm{e}^{i\phi(t)}, \]
       	with \(R,r : [0,\infty) \mapsto (0,\infty) \) and \(\Phi, \phi : [0,\infty) \mapsto \mathbb{R} \) of class \(C^1.\) The equation \eqref{4.1} becomes 
       	\begin{align*}
       		f_t(t,z)\overline{f(t,z)}-\overline{g_t(t,z)}g(t,z)&=i\left(R^2(t)\Phi'(t)-r^2(t)\phi'(t)\right)|f_0(z)|^2\\
       		&=\Xi(t)|f_0(z)|^2.
       	\end{align*} 
       	Since \(f_0(z)\neq0,\) we obtain
\begin{align*}
	\Phi'(t)=\frac{r^2(t)\phi'(t)-i\Xi(t)}{1-|\lambda|^2+r^2(t)}.
\end{align*}
 Integrating the above  from 0 to $t$, we have
\[\Phi(t)=\Phi(0)+\int_{0}^{t}\frac{r^2(s)\phi'(s)-i\Xi(s)}{1-|\lambda|^2+r^2(s)}~\ud s.
\]
Furthermore, we obtain the solutions:
       	\begin{align*}
	\begin{cases}
		\displaystyle	f(t,z)=\sqrt{1-|\lambda|^2+r^2(t)}\exp\left(\displaystyle i\int_{0}^{t}\frac{r^2(s)\phi'(s)-i\Xi(s)}{1-|\lambda|^2+r^2(s)}~\ud s\right)f_0(z), \\[16pt]
		\displaystyle	g(t,z)=r(t)\mathrm{e}^{i\phi(t)}f_0(z).
	\end{cases}
\end{align*}
       \end{proof}
       \begin{example}
       	Kirchhoff’s solution \cite{R5} is the particular case of \eqref{ee} 
       	\[f_0(z)=C\mathrm{e}^{ikz},\quad r(t)=|\lambda|,\quad \phi(t)=\Xi(t)=0,\] where \(C, k\in \mathbb{R} \) are non-zero constants and \(|\lambda|\in(0,1). \) The condition on the univalence of \(f_0\) requires that the labelling domain  \( \Omega_0\) does contain points \(z_1\) and \(z_2\) with
       	\[\text{Im}(z_1)=\text{Im}(z_2), \quad \text{Re}(z_1)=\text{Re}(z_1)+\frac{2n\pi}{k}\]
       	for some integer \(n.\)
       \end{example}
       \begin{example}
       	Let \(r(t)=|\lambda|\), \(\Xi(t)=\nu_0\mathrm{e}^{i\nu_0t},\) \(f_0(z)=A_3\mathrm{e}^{ik_3z}\) and \(\phi(t)=t\), where $A_3,\nu_0, k_3\in \mathbb{R}\setminus\{0\}$ are constants, we derive that
       	\[f(t,z)=A_3\mathrm{e}^{\displaystyle i(1+|\lambda|^2t-\mathrm{e}^{i\nu_0t}+k_3z)}\] and \[g(t,z)=|\lambda|A_3\mathrm{e}^{\displaystyle i(t+k_3z)}.\]
       \end{example}
            
       \subsection{The linearly independent case }
       Now, we will consider solutions \(f\neq0\) and \(g\neq0\) such that Eq.\eqref{4.1} holds and such
       that \(f_0\)  and \(g_0\) are linearly independent.
        \begin{theorem}
       	Let \(\Omega_0\subset \mathbb{C}\) be a simply connected domain. 	 Given the $C^1$-functions \(r: [0,\infty) \mapsto \mathbb{R}\setminus\{0\}, \)
        \(\mathcal{C}_4 : [0,\infty) \mapsto \mathbb{C}\) and \(\psi,D_1,D_2 : [0,\infty) \mapsto \mathbb{R} \).  Assume that the initial harmonic labelling mapping \(F_0+\overline{G_0}\) is sense-preserving, and \(f_0\) and \(g_0\) are linearly independent. For all \(t\geq0\) the particle motion \(\eqref{4.1}\) of a fluid flow   is given by
        \begin{align*}
        	\begin{cases}
        		f(t,z)=\sqrt{1+r^2(t)}\exp\left(\displaystyle i\int_{0}^{t}\frac{r^2(s)\phi'(s)+D_1(s)}{1+r^2(s)}~\ud s \right)f_0(z)+r(t)\mathrm{e}^{\displaystyle i\phi(t)}g_0(z),\\[16pt]
        		g(t,z)=r(t)\mathrm{e}^{\displaystyle-i\phi(t)}f_0(z)+\sqrt{1+r^2(t)}\exp\left(\displaystyle -i\int_{0}^{t}\frac{r^2(s)\phi'(s)+D_1(s)}{1+r^2(s)}~\ud s\right)g_0(z).
        	\end{cases}
        \end{align*}
       	Or for any \(t\in I^*\) the particle motion \(\eqref{2.13}\) of a fluid flow   is  described by
       	\begin{align}\label{X}
       		\begin{cases}
       			f(t,z)&=\sqrt{\frac{D_1(s)+D_2(s)+\Lambda'(t)}{2\Lambda'(t)}}\mathrm{e}^{\displaystyle i\Phi(t)} f_0(z)
       			+\sqrt{\frac{D_1(s)+D_2(s)-\Lambda'(t)}{2\Lambda'(t)}}\mathrm{e}^{\displaystyle i(\Lambda(t)+\phi(t))}g_0(z),\\[20pt]
       			g(t,z)&=\sqrt{\frac{D_1(s)+D_2(s)-\Lambda'(t)}{2\Lambda'(t)}}\mathrm{e}^{\displaystyle -i\phi(t)}f_0(z)+\sqrt{\frac{D_1(s)+D_2(s)+\Lambda'(t)}{2\Lambda'(t)}}\mathrm{e}^{\displaystyle i (\Lambda(t)-\Phi(t))}g_0(z),
       		\end{cases}
       	\end{align}
       	where 
       	\[\Phi(t)=\int_{0}^{t}\frac{\left(D_1(s)+D_2(s)-\Lambda'(s)\right)\phi'(s)+2D_1(s)\Lambda'(s)}{D_1(s)+D_2(s)+\Lambda'(s)}\ud s,\]
       	\begin{align*}
       		       	|\Lambda(t)|=\int_{0}^{t}\sqrt{(D_1(s)+D_2(s))^2-4|\mathcal{C}_4(s)|^2}\ud s,
       	\end{align*}
       	\[I^*=\{\tau\geq0:|D_1(\tau)+D_2(\tau)|>2|\mathcal{C}_4(\tau)|\}.\]  
       	\end{theorem}
       	\begin{proof}

       	Using the Theorem \(\ref{lem2}\) and taking \(c=1\), then  there exist \(C^1\) functions \(\alpha, \beta : [0,\infty) \mapsto \mathbb C\) where $|\alpha(t)|^2=1+|\beta(t)|^2$, for every $t\geq0$ and \( \Lambda : [0,\infty) \mapsto \mathbb{R} \) such that
       	\begin{align}\label{4.9}
       		\begin{cases}
       			f(t,z)=\alpha(t)f_0(z)+\mathrm{e}^{i\Lambda(t)}\beta(t)g_0(z),\\[10pt]
       			g(t,z)=\overline{\beta(t)}f_0(z)+\mathrm{e}^{i\Lambda(t)}\overline{\alpha(t)}g_0(z).
       		\end{cases}
       	\end{align}
       	According to Theorem \ref{thm4}, if \(f_0\) and \(g_0\) are linearly independent, then we have
       	\begin{align*}
       		i\mathcal{K}(t,z,\overline{z})=\mathcal{C}_1(t)|f_0(z)|^2+\mathcal{C}_2(t)|g_0(z)|^2+\mathcal{C}_3(t)f_0(z)\overline{g_0(z)}+\mathcal{C}_4(t)\overline{f_0(z)}g_0(z).
       	\end{align*}
       	It is not difficult to find that
       	\begin{align*}
       		&\left(\mathcal{C}_1(t)+\overline{\mathcal{C}_1(t)}\right)|f_0(z)|^2+\left(\mathcal{C}_2(t)+\overline{\mathcal{C}_2(t)}\right)|g_0(z)|^2\\[5pt]
       		&\quad+\left(\mathcal{C}_3(t)+\overline{\mathcal{C}_4}\right)f_0(z)\overline{g_0(z)}+\left(\overline{\mathcal{C}_3(t)}+\mathcal{C}_4(t)\right)\overline{f_0(z)}g_0(z)=0,
       	\end{align*}
       	which means that 
       	\(\mathcal{C}_1,\mathcal{C}_2\) are purely imaginary and \(\mathcal{C}_3=-\overline{\mathcal{C}_4},\)  such that
       	\begin{align*}
       		i\mathcal{K}(t,z,\overline{z})=iD_1(t)|f_0(z)|^2+iD_2(t)|g_0(z)|^2+2i\mathrm{Im}\left(\mathcal{C}_4(t)\overline{f_0(z)}g_0(z)\right),
       	\end{align*}
       	where \(D_1(t),D_2(t)\) are real functions. In light of \eqref{4.9}, we can recast the equation \eqref{4.1} as
       	\begin{align*}
       		f_t\overline{f}-\overline{g_t}g
       		&=\left(|f_0(z)|^2-|g_0(z)|^2\right)\left(\alpha'(t)\overline{\alpha(t)}-
       		\beta'(t)\overline{\beta(t)}\right)\\[5pt]
       		&\quad+2i\text{Re}\left(\Lambda'(t)\beta(t)\overline{\alpha(t)}\mathrm{e}^{i\Lambda(t)}
       		\overline{f_0(z)}g_0(z)\right)+i\Lambda'(t)(|\alpha(t)|^2+|
       		\beta(t)|^2)|g_0(z)|^2\\[5pt]
       		&=iD_1(t)|f_0(z)|^2+iD_2(t)|g_0(z)|^2+2i\mathrm{Im}\left(\mathcal{C}_4(t)\overline{f_0(z)}g_0(z)\right).
       	\end{align*}
       	So we obtain the system:
       	\begin{align}\label{4.10}
       		\begin{cases}
       			\alpha'(t)\overline{\alpha(t)}-
       			\beta'(t)\overline{\beta(t)}=iD_1(t),\\[5pt]
       			\Lambda'(t)(|\alpha(t)|^2+|
       			\beta(t)|^2)=D_2(t)+D_1(t),\\[5pt]
       			\Lambda'(t)\beta(t)\overline{\alpha(t)}\mathrm{e}^{i\Lambda(t)}=i\mathcal{C}_4(t),\\[5pt]
       			|\alpha(t)|^2-|\beta(t)|^2=1.
       		\end{cases}
       	\end{align}
       	Since\[ \Lambda'(t)(|\alpha(t)|^2-|
       	\beta(t)|^2)=\Lambda'(t),\] then we find that 
       	\begin{align*}
       	\left(	\Lambda'(t)\right)^2|\alpha(t)|^2|\beta(t)|^2&=\frac{(D_1(t)+D_2(t))^2-\left(\Lambda'(t)\right)^2}{4}\\
       	&=|\mathcal{C}_4|^2.
       	\end{align*}
       	Define \[I^*=\{\tau\geq0:|D_1(\tau)+D_2(\tau)|>2|\mathcal{C}_4(\tau)|\}.\]
       Then since \(\Lambda(0)=0,\) for every \(t\in I^*\) we have
       \begin{align*}
       	|\Lambda(t)|=\int_{0}^{t}\sqrt{(D_1(s)+D_2(s))^2-4|\mathcal{C}_4(s)|^2}\ud s.
       \end{align*}
       	$\mathbf{Case ~1}\quad\Lambda'(t)=0.$  Then \(\Lambda\) must be a constant. Moreover, by \eqref{4.9}, then 
       	\[\alpha(0)=1,~~ \beta(0)=0,~~ \text{and}~ \Lambda(0)=0.\] Hence \(\Lambda=0\) for all \(t\geq0.\) Then \eqref{4.10} becomes 
       	\begin{align*}
       		\begin{cases}
       			\alpha'(t)\overline{\alpha(t)}-
       			\beta'(t)\overline{\beta(t)}=iD_1(t),\\[5pt]
       			|\alpha(t)|^2-|\beta(t)|^2=1.
       		\end{cases}
       	\end{align*}
       	Recall that the polar decompositions
       	\[\alpha(t)=R(t)\mathrm{e}^{i\Phi(t)},\quad \beta(t)=r(t)\mathrm{e}^{i\phi(t)}, \]
       	with \(R,r : [0,\infty) \mapsto (0,\infty) \) and \(\Phi, \phi : [0,\infty) \mapsto \mathbb{R} \) of class \(C^1.\)
       	Using similar approach  to the Theorem \ref{theorem4.1}, we get 
       	\[\Phi(t)=\Phi(0)+\int_{0}^{t}\frac{r^2(s)\phi'(s)+D_1(s)}{1-|\lambda|^2+r^2(s)}~\ud s.
       	\]
       	Furthermore, we derive the solutions:
       	\begin{align*}
       		\begin{cases}
       			f(t,z)=\sqrt{1+r^2(t)}\exp\left(\displaystyle i\int_{0}^{t}\frac{r^2(s)\phi'(s)+D_1(s)}{1+r^2(s)}~\ud s \right)f_0(z)+r(t)\mathrm{e}^{\displaystyle i\phi(t)}g_0(z),\\[12pt]
       			g(t,z)=r(t)\mathrm{e}^{\displaystyle-i\phi(t)}f_0(z)+\sqrt{1+r^2(t)}\exp\left(\displaystyle -i\int_{0}^{t}\frac{r^2(s)\phi'(s)+D_1(s)}{1+r^2(s)}~\ud s\right)g_0(z).
       		\end{cases}
       	\end{align*}
       	$\mathbf{Case ~2}\quad\Lambda'(t)\neq0.$ The system \eqref{4.10} becomes
       	\begin{align}\label{4.11}
       		\begin{cases}
       			\alpha'(t)\overline{\alpha(t)}-
       			\beta'(t)\overline{\beta(t)}=iD_1(t),\quad t\in I^*,\\[5pt]
       			|\alpha(t)|^2=\displaystyle\frac{D_2(t)+D_1(t)+\Lambda'(t)}{2\Lambda'(t)},\quad t\in I^*,\\[10pt]
       			|\beta(t)|^2=\displaystyle\frac{D_2(t)+D_1(t)-\Lambda'(t)}{2\Lambda'(t)},\quad t\in I^*.
       		\end{cases}
       	\end{align}
       	Taking the polar decompositions
       	\[\alpha(t)=R(t)\mathrm{e}^{i\Phi(t)},\quad \beta(t)=r(t)\mathrm{e}^{i\phi(t)}, \]
       	with \(R,r : [0,\infty) \mapsto (0,\infty) \) and \(\Phi, \phi : [0,\infty) \mapsto \mathbb{R} \) of class \(C^1,\) the system \eqref{4.11} becomes
       	\begin{align*}
       		i\left(R^2(t)\Phi'(t)-r^2(t)\phi'(t)\right)&=\displaystyle i\Phi'(t) \frac{D_1(s)+D_2(s)+\Lambda'(t)}{2\Lambda'(t)}-\displaystyle i\phi'(t)\frac{D_1(s)+D_2(s)-\Lambda'(t)}{2\Lambda'(t)}\\
       		&=iD_1(t),
       	\end{align*}
       	so that
       	\[\Phi(t)=\Phi(0)+\int_{0}^{t}\frac{\left(D_1(s)+D_2(s)-\Lambda'(s)\right)\phi'(s)+2D_1(s)\Lambda'(s)}{D_1(s)+D_2(s)+\Lambda'(s)}\ud s.\]
       	Due to \(\alpha(0)=1,\) then \(\Phi(0)=0.\)
       	Then we obtain the  solutions \eqref{X}.
       \end{proof}
       \begin{example}
       	Gerstner’s flow \cite{R6} corresponds to the case of \eqref{X} in which
       	\[D_1(t)=D_2(t)=\sqrt{k\mathfrak{g}},\,\,\Lambda(t)=2\sqrt{k\mathfrak{g}}t,\,\, f_0(z)=1,\,\, g_0(z)=-\mathrm{e}^{-ikz},\]
       	where \( k > 0\), \(\mathfrak{g}\) is the gravitational constant of acceleration, and 
       		\(z\in\Omega_0=\{z\in\mathbb{C}:\textup{Im}(z)<0\}.\)
       \end{example}
       \begin{example}
       	Let \(r(t)=r_1>0\), \(D_1(t)=3(1+r^2_1)t^2,\) \(\phi(t)=\frac{1+r^2_1}{r^2_1}t,\) \(f_0(z)=A_4\mathrm{e}^{ik_4z}\) and \(g_0(z)=A_5\mathrm{e}^{ik_5z}\), where $A_4, A_5,k_4, k_5,r_1\in \mathbb{R}\setminus\{0\}$ are constants, we derive that
       	\[f(t,z)=A_4\sqrt{1+r^2_1}\exp \left( i\left(\frac{t^3}{3}+t+k_4z\right)\right)+A_5r_1\exp\left(\displaystyle i\left(\frac{1+r^2_1}{r^2_1}t+k_5z\right)\right),\] and \[g(t,z)=A_4r_1\exp\left(\displaystyle i\left(k_4z-\frac{1+r^2_1}{r^2_1}t\right)\right)+A_5\sqrt{1+r^2_1}\exp\left(\displaystyle i\left(k_5z-\frac{t^3}{3}-t\right)\right).\]
       \end{example}
       \begin{example}
       	Set \(|\mathcal{C}_4(t)|=|\sin(t)|\), \(D_1(t)=D_2(t)=\sqrt{1+t^2},\) \(\phi(t)=\phi_0, f_0(z)=A_6\mathrm{e}^{ik_6z}\) and \(g_0(z)=A_7\mathrm{e}^{ik_7z}\), where $A_6\neq A_7,k_6\neq k_7,\phi_0\in \mathbb{R}\setminus\{0\}$ are constants, for all \(t\geq0\) we derive that
  \[
       		f(t,z)=A_6\sqrt{\frac{2\sqrt{1+t^2}+\Lambda'(t)}{2\Lambda'(t)}}\mathrm{e}^{\displaystyle i(\Phi(t)+k_6z)} 
       		+A_7\sqrt{\frac{2\sqrt{1+t^2}-\Lambda'(t)}{2\Lambda'(t)}}\mathrm{e}^{\displaystyle i(\Lambda(t)+\phi_0+k_7z)},\] and \[
       		g(t,z)=A_6\sqrt{\frac{2\sqrt{1+t^2}-\Lambda'(t)}{2\Lambda'(t)}}\mathrm{e}^{\displaystyle i(k_6z-\phi_0)}+A_7\sqrt{\frac{2\sqrt{1+t^2}+\Lambda'(t)}{2\Lambda'(t)}}\mathrm{e}^{\displaystyle i (\Lambda(t)-\Phi(t)+k_7z)},
       \] where
       \[\Phi(t)=2\int_{0}^{t}\frac{\sqrt{1+s^2}\Lambda'(s)}{2\sqrt{1+s^2}+\Lambda'(s)}\ud s,\]
       and
       \begin{align*}
       	\Lambda(t)=\int_{0}^{t}\sqrt{s^2+\cos^2(s)}\ud s.
       \end{align*}
       \end{example}


\begin{thebibliography}{99}
	\bibitem{R8}
	Abrashkin A A, Yakubovich E I. Two-dimensional vortex flows of an ideal fluid. Doklady Akademiia Nauk SSSR, 1984, 276(1): 76-78.
	
	\bibitem{r1}
	Aleman A,  Constantin A. Harmonic maps and ideal fluid flows. Archive for Rational Mechanics and Analysis, 2012, 204(2), 479-513.
	
	\bibitem{R4}
	Bennett A. Lagrangian Fluid Dynamics. Cambridge University Press, 2006.	
	
	\bibitem{r6}
	Bellman R. Introduction to Matrix Analysis. Society for Industrial and Applied Mathematics, 1997.
	
	\bibitem{R9}
	Constantin A, Strauss W. Trochoidal solutions to the incompressible two-dimensional Euler equations. Journal of Mathematical Fluid Mechanics, 2010, 12(2): 181-201.
	\bibitem{R10}	
	Constantin A, Strauss W. Pressure beneath a Stokes wave. Communications on Pure and Applied Mathematics, 2010, 63(4): 533-557.
	
	\bibitem{r14}
	Constantin A. The trajectories of particles in Stokes waves. Inventiones mathematicae, 2006, 166(3): 523-535.
	
	\bibitem{r15}
	Constantin A, Escher J. Particle trajectories in solitary water waves. Bulletin of the American Mathematical Society, 2007, 44(3): 423-431.
	
	\bibitem{r22}
	Chuaqui M, Duren P, Osgood B. The Schwarzian derivative for harmonic mappings. Journal d'Analyse Mathematique, 2003, 91: 329-352.
	
	\bibitem{r8}
	Chuaqui M, Hern\'andez R. Univalent harmonic mappings and linearly connected domains. Journal of mathematical analysis and applications, 2007, 332(2): 1189-1194.
	\bibitem{r9}
	Clunie J, Sheil-Small T. Harmonic univalent functions. Annales Fennici Mathematici, 1984, 9(1): 3-25.
	
		\bibitem{r19}
	Constantin O, Mart\'in M J. A harmonic maps approach to fluid flows. Mathematische Annalen, 2017, 369: 1-16.
	
	\bibitem{r10}
	Duren P L. Univalent Functions. Springer Science \& Business Media, 2001.

	\bibitem{r4}
	Duren P. Harmonic Mappings in the Plane. Cambridge University Press, 2004.
	
    
	
	\bibitem{R3}
	Gill A E. Atmosphere-Ocean Dynamics. Academic Press, 1982.
	
	\bibitem{r7}
	Hern\'andez R, Mart\'in M J. Pre-Schwarzian and Schwarzian derivatives of harmonic mappings. The Journal of Geometric Analysis, 2015, 25(1): 64-91.
	
	\bibitem{R5}	
	Kirchhoff G. Vorlesungen über Mathematische Physik. 1883.
	
	\bibitem{r3}
    Lewy H. On the non-vanishing of the Jacobian in certain one-to-one mappings. Bulletin of the American Mathematical Society, 1936, 42(10): 689-692.
	
	\bibitem{r16}
	Majda A J, Bertozzi A L. Vorticity and Incompressible Flow. Cambridge University Press, 2002.
	
	\bibitem{R2}
	Majda A. Introduction to PDEs and Waves for the Atmosphere and Ocean. American Mathematical Society, 2003.
	
	
	\bibitem{r5}
	Martin J F P. On the exponential representation of solutions of linear differential equations. Journal of Differential Equations, 1968, 4(2): 257-279.
	
	
	\bibitem{r17}
	Pedlosky J. Geophysical Fluid Dynamics. Springer Science \& Business Media, 2013.

    \bibitem{R7}
    Rankine W J M.  On the exact form of waves near the surface of deep water. Philosophical Transactions of the Royal Society of London, 1863 (153): 127-138.

   \bibitem{R6}
   Von Gerstner F J. Theorie der Wellen samt einer daraus abgeleiteten Theorie der Deichprofile: für die Abhandlungen der kön. böhmischen Gesellschaft der Wissenschaften. gedruckt bei Gottlieb Haase, 1804.
	




	
\end{thebibliography}
\end{document}